\documentclass{elsarticle}
\journal{arXiv}

\usepackage{amssymb}

\usepackage{amsmath}
\usepackage{amsfonts}
\usepackage{amsthm}
\usepackage{mathrsfs} 
\usepackage{color}
\usepackage{xcolor}
\usepackage{graphicx}
\usepackage{tikz}
\usepackage{tikz-cd}
\usetikzlibrary{arrows}
\usepackage{enumitem} 
\usepackage{mathtools} 
\usepackage{hyperref}
\usepackage{comment}
\usepackage{multicol}
\usepackage{turnstile}
\usepackage[edges]{forest}
\usepackage{caption}
\usepackage{subcaption}
\usepackage{rotating}

\forestset{
declare count register=labelcount,
enumerate/.style={labelcount=#1,
for tree={content/.pgfmath=labelcount, labelcount-=1},
}
}

\newtheorem{thm}{Theorem}[section]
\newtheorem{prop}[thm]{Proposition}
\newtheorem{lemma}[thm]{Lemma}
\newtheorem{cor}[thm]{Corollary}

\theoremstyle{definition}

\newtheorem{remark}[thm]{Remark}

\newcommand{\set}[1]{\{ #1 \}}
\newcommand{\meet}{\mathbin{\wedge}}
\newcommand{\join}{\mathbin{\vee}}

\newcommand{\alg}[1]{\langle #1 \rangle}
\newcommand{\pair}[1]{\langle #1 \rangle}
\newcommand{\con}{\mathrm{Con}}

\newcommand{\rot}[1]{{{#1}^r}}

\renewcommand{\phi}{\varphi}
\renewcommand{\epsilon}{\varepsilon}
\newcommand{\hm}{\mathbb{H}}
\newcommand{\iso}{\mathbb{I}}

\newcommand{\pu}{\mathbb{P}_\mathrm{U}}
\newcommand{\sub}{\mathbb{S}}
\newcommand{\hspu}{\mathbb{HSP}_\mathrm{U}}

\newcommand{\vr}{\mathbb{V}}
\newcommand{\m}[1]{ {\bf #1}}

\newcommand{\V}{\mathsf{V}}
\newcommand{\W}{\mathsf{W}}
\newcommand{\K}{\mathsf{K}}
\newcommand{\Vfsi}{\mathsf{V}_{\mathrm{FSI}}}
\newcommand{\Vfc}{\mathsf{V}_{\mathrm{fc}}}
\newcommand{\Wfc}{\mathsf{W}_{\mathrm{fc}}}

\newcommand{\cls}[1]{\mathsf{#1}}
\newcommand{\wajs}{\mathrm{Wajs}}
\newcommand{\luk}{\mathrm{\textup{Ł}}}
\newcommand{\basic}{\mathrm{Basic}}
\newcommand{\basicfc}{\mathrm{Basic}_\mathrm{fc}}
\newcommand{\bvarint}[1]{\mathcal{I}(#1)}

\newcommand{\BHAP}{\mathcal{BH}_{\mathrm{AP}}}
\newcommand{\BL}{\mathsf{BL}}
\newcommand{\BH}{\mathsf{BH}}

\makeatletter
\providecommand*{\leftmodels}{%
  \mathrel{%
    \mathpalette\@leftmodels\models
  }%
}
\newcommand*{\@leftmodels}[2]{%
  \reflectbox{$\m@th#1#2$}%
}
\makeatother

\newcommand{\rad}{\mathrm{Rad}}
\newcommand{\Rad}{\mathbf{Rad}}

\newcommand{\Sm}[1]{\pmb{\textup{\L}}_{#1}}
\newcommand{\Smw}[1]{\pmb{\textup{\L}}_{{#1},\omega}}
\newcommand{\Cn}[1]{\m{W}_{#1}}
\newcommand{\Cnw}[1]{\m{W}_{{#1},\omega}}
\newcommand{\Cw}{\m{Z}}

\newcommand{\WH}{\cls{WH}}
\newcommand{\MV}{\cls{MV}}

\newcommand{\wuz}{$[\Cn{n}]\cup [ \Cw]$}
\newcommand{\wsuz}{$[\Cn{n}^\ast]\cup [ \Cw]$}
\newcommand{\wsuzs}{$[ \Cn{n}^\ast]\cup [ \Cw^\ast]$}
\newcommand{\wszs}{$[ \Cn{n}^\ast \Cw^\ast]$}
\newcommand{\wzd}{$[ (\Cn{n} \Cw)^\ast]$}
\newcommand{\zws}{$[\Cw \Cn{n}^\ast]$}
\newcommand{\zsws}{$[ \Cw^\ast \Cn{n}^\ast]$}
\newcommand{\wsz}{$[\Cn{n}^\ast \Cw]$}
\newcommand{\wz}{$[ \Cn{n} \Cw]$}
\newcommand{\wzs}{$[ \Cn{n} \Cw^\ast]$}
\newcommand{\zw}{$[ \Cw \Cn{n}]$}
\newcommand{\zsw}{$[ \Cw^\ast \Cn{n}]$}
\newcommand{\wuzs}{$[\Cn{n}]\cup [  \Cw^\ast]$}

\newcommand{\mwuz}{[\Cn{n}]\cup [ \Cw]}
\newcommand{\mwsuz}{[\Cn{n}^\ast]\cup [ \Cw]}
\newcommand{\mwsuzs}{[ \Cn{n}^\ast]\cup [ \Cw^\ast]}
\newcommand{\mwszs}{[ \Cn{n}^\ast \Cw^\ast]}
\newcommand{\mwzd}{[ (\Cn{n} \Cw)^\ast]}
\newcommand{\mzws}{[\Cw \Cn{n}^\ast]}
\newcommand{\mzsws}{[ \Cw^\ast \Cn{n}^\ast]}
\newcommand{\mwsz}{[\Cn{n}^\ast \Cw]}
\newcommand{\mwz}{[ \Cn{n} \Cw]}
\newcommand{\mwzs}{[ \Cn{n} \Cw^\ast]}
\newcommand{\mzw}{[ \Cw \Cn{n}]}
\newcommand{\mzsw}{[ \Cw^\ast \Cn{n}]}
\newcommand{\mwuzs}{[\Cn{n}]\cup [  \Cw^\ast]}

\begin{document}

\begin{frontmatter}



\title{Interpolation in H\'{a}jek's Basic Logic}

 \author[label1]{Wesley Fussner}
 \affiliation[label1]{organization={Institute of Computer Science, Czech Academy of Sciences},
             country={Czechia}}

 \author[label2]{Simon Santschi\fnref{fn2}}
 \affiliation[label2]{organization={Mathematical Institute, University of Bern},
             country={Switzerland}}
\fntext[fn1]{Supported by Swiss National Science Foundation (SNSF) grant No. 200021\textunderscore 215157.}




\begin{abstract}
We exhaustively classify varieties of BL-algebras with the amalgamation property, showing that there are only countably many of them and solving an open problem of Montagna. As a consequence of this classification, we obtain a complete description of which axiomatic extensions of H\'{a}jek's basic fuzzy logic $\m{BL}$ have the deductive interpolation property. Along the way, we provide similar classifications of varieties of basic hoops with the amalgamation property and axiomatic extensions of the negation-free fragment of $\m{BL}$ with the deductive interpolation property.
\end{abstract}




\end{frontmatter}



\section{Introduction}

H\'{a}jek introduced the system of basic logic $\m{BL}$ in the late 1990s as a mathematical foundation for theretofore existing approaches to fuzzy inference. Basic logic's fundamental importance to fuzzy reasoning was cemented shortly after its introduction, when in \cite{CEGT2000} it was shown that $\m{BL}$ is exactly the logic of continuous triangular norms: The theorems of $\m{BL}$ coincide with the formulas validated in every algebra of the form $\alg{[0,1],\cdot,\to, 0, 1}$, where $[0,1]$ is the real unit interval, $\cdot$ is a continuous triangular norm on $[0,1]$, and $\to$ is its residuum.\footnote{The connectives $\meet$, $\join$ are definable from $\cdot$ and $\to$, so the result also holds for expanded language including these.} In the intervening two and a half decades, $\m{BL}$ has been the subject of an impressive array of studies and, thanks to these, we now know the answers to most of the fundamental logical questions one might ask in this vicinity. For instance, the problem of recognizing theorems in $\m{BL}$ was shown in \cite{BHMV2001} to be coNP-complete (hence decidable) and effective proof search methods were given in \cite{BM2008}. On the semantic side, the structure of algebraic models of $\m{BL}$ (viz. \emph{BL-algebras}) has been described in extensive detail (see, e.g., \cite{Agliano2003,AgBova2010}) and transparent relational semantics for $\m{BL}$ and its extensions were given in \cite{Fussner2022,ABM2009}. Among other things, semantic work has yielded G\"{o}del-style modal translations for $\m{BL}$ and its fragments \cite{FZ2021}, detailed descriptions of the lattice of axiomatic extensions of $\m{BL}$ \cite{Agliano2019}, and a complete classification of axiomatic extensions of $\m{BL}$ with the deductive Beth definability property \cite{Montagna2006}. It is fair to say that our understanding of $\m{BL}$ is now quite mature.

One problem that has stubbornly resisted resolution, however, is that of obtaining a complete grasp of interpolation in $\m{BL}$ and its extensions. \cite[Theorem~6.1]{Montagna2006} shows that there are only three consistent axiomatic extensions $\m{L}$ of $\m{BL}$ with the \emph{Craig interpolation property}:
\begin{quote}
Whenever $\varphi\to\psi$ is a theorem of $\m{L}$, there exists a formula $\chi$, whose variables appear in both $\varphi$ and $\psi$, such that both $\varphi\to\chi$ and $\chi\to\psi$ are theorems of $\m{L}$.
\end{quote}
On the other hand, among substructural logics with exchange---of which \m{BL} is an example---the Craig interpolation property is, in general, weaker than the \emph{deductive interpolation property}. The latter may be formulated as follows, where we use $\vdash_\m{L}$ to denote the consequence relation of the logic $\m{L}$:
\begin{quote}
Whenever $\varphi\vdash_\m{L}\psi$, there exists a formula $\chi$, whose variables appear in both $\varphi$ and $\psi$, such that both $\varphi\vdash_\m{L}\chi$ and $\chi\vdash_\m{L}\psi$.
\end{quote}
The Craig interpolation property is equivalent to the deductive interpolation property among superintuitionistic logics, and the Craig interpolation property implies the deductive interpolation property for any axiomatic extension of $\m{BL}$. However, there are many extension of $\m{BL}$ that have deductive interpolation but not Craig interpolation.

Montagna was the first to make a significant investigation of deductive interpolation for $\m{BL}$ and its extensions. His study \cite{Montagna2006} used universal-algebraic methods, exploiting the fact that an axiomatic extension of $\m{BL}$ has the deductive interpolation property if and only if the corresponding equational class of BL-algebras modeling it has the amalgamation property. Among other results, \cite{Montagna2006} showed that $\m{BL}$ itself---as well as many of its most salient extensions---have the deductive interpolation property, but that uncountably many axiomatic extensions of $\m{BL}$ fail to have deductive interpolation. Montagna was unable to characterize the extensions of $\m{BL}$ with deductive interpolation, however, and identified this as an important open problem for $\m{BL}$ at \cite[p.~178]{Montagna2006}.

Since the publication of \cite{Montagna2006}, there have been quite a few studies of deductive interpolation in extensions of $\m{BL}$. Coronesi, Marchioni, and Montagna applied methods from quantifier elimination to study amalgamation in BL-algebras in \cite{CMM11}. Later, going back to more purely algebraic methods, Aguzzoli and Bianchi in \cite{AB21} offered a partial classification of varieties of BL-algebras that have the amalgamation property. However, this partial classification was limited by certain technical finiteness hypotheses on the varieties considered. Subsequently, building on another approach articulated in \cite{Fussner2023}, Aguzzoli and Bianchi in \cite{AB23} sharpened their partial classification from \cite{AB21} but could not eliminate the finiteness hypotheses inherent in \cite{AB21}. None of the aforementioned studies have provided a complete classification of the axiomatic extensions of $\m{BL}$ with the deductive interpolation property. 

The present paper gives just such a classification, providing a solution to Montagna's problem posed at \cite[p.~178]{Montagna2006}. Like many of the previously cited studies, our work is algebraic and exploits the link between the deductive interpolation property and the amalgamation property. Also like many of the aforementioned studies, we heavily rely on the decomposition of arbitrary totally ordered BL-algebras as ordinal sums of the $0$-free subreducts of MV-algebras, often called Wajsberg hoops \cite{Agliano2003}. In fact, much of our analysis will focus on basic hoops, the $0$-free subreducts of BL-algebras, and we will also provide a complete classification for varieties of these that have the amalgamation property.

Our approach departs from the previous attempts at classifying extensions $\m{BL}$ with the deductive interpolation property in two important respects, however. First, we introduce a streamlined condition for studying the amalgamation property---called the \emph{essential amalgamation property}---that is tailored to varieties generated by totally ordered algebras, of which BL-algebras and basic hoops are examples. Second, we give a new nomenclature for varieties of basic hoops that is inspired by regular expressions, and which proves itself to be a valuable organizational tool for identifying varieties with the amalgamation property. By applying these two new ideas in tandem, we obtain a classification of which varieties of basic hoops have the amalgamation property and, in turn, use our classification for basic hoops to obtain an exhaustive description of which varieties of BL-algebras have the amalgamation property. Via the link between axiomatic extensions of $\m{BL}$ and varieties of BL-algebras, this provides a complete description of which axiomatic extensions of $\m{BL}$ have the deductive interpolation property. This description turns out to be unexpectedly tangible: Concretely, we show that the poset of varieties of basic hoops with the amalgamation property may be written as the disjoint union of countably many finite intervals, each of which we explicitly describe. A similar---albeit somewhat more complicated---description is given for BL-algebras. In particular, this shows that there are only countably many varieties of BL-algebras with the amalgamation property.

The paper proceeds as follows. In Section~\ref{sec:preliminaries}, we lay out the background on BL-algebras that we will need throughout the paper. In particular, we recall the definitions of several algebraic constructions---notably the ordinal sum and disconnected rotation---that we will invoke repeatedly later on. Then, in Section~\ref{sec:essential}, we introduce a new equivalent formulation of the amalgamation property that reduces the study of amalgamation in varieties to more tractable generating algebras. The presentation of the amalgamation property that we obtain there is very well adapted to studying amalgamation in varieties generated by totally ordered algebras, and we anticipate that it will have many applications elsewhere. In Section~\ref{sec:hoops} we provide an exhaustive description of the varieties of basic hoops with the amalgamation property. Although necessarily technical, our description is nevertheless transparent: We show that the poset of varieties of basic hoops with the amalgamation property may be partitioned into a countably infinite family of finite intervals, describing these explicitly. As a consequence, we obtain that there are only countably many varieties of basic hoops with the amalgamation property. In Section~\ref{sec:BL-algebras}, we use our classification of varieties of basic hoops with the amalgamation property---along with the well-known classification of varieties of MV-algebras with amalgamation \cite{DiNola2000}---to obtain a similar classification for varieties of BL-algebras. This classification is the main result of the paper, and in particular entails that there are only countably many varieties of BL-algebras that have the amalgamation property. Subsequently, in Section~\ref{sec:logic}, we apply the results of Section~\ref{sec:hoops} and \ref{sec:BL-algebras} to obtain complete classifications of axiomatic extensions of $\m{BL}$ with the deductive interpolation property, as well as a similar classification for the negation-free fragment of $\m{BL}$. Finally, in Section~\ref{sec:conclusion} we give some concluding remarks, situating these results among other recent progress on interpolation in substructural logics.

\section{Preliminaries}\label{sec:preliminaries}
A \emph{commutative residuated lattice} is an algebra $\alg{A,\meet,\join,\cdot,\to,1}$, where $\alg{A,\cdot,1}$ is a commutative monoid, $\alg{A,\meet,\join}$ is a lattice, and for all $x,y,z\in A$,
\[
x\cdot y\leq z \iff x\leq y\to z,
\]
where $\leq$ is the lattice ordered defined by $x\leq y$ if and only if $x\meet y = x$. In the event that $1$ is the greatest element of a commutative residuated lattice $\m{A}$, we say that $\m{A}$ is \emph{integral}. If, in addition, $\m{A}$ has a least element $0$, we say that $\m{A}$ is \emph{bounded}. Sometimes we will designate this least element in the signature, and it is important in this study to distinguish between the cases when we do so and the cases when we do not. In the event that the least element of an integral commutative residuated lattice $\m{A}$ is designated in the signature, we say that $\m{A}$ has \emph{designated bounds}. As usual, we will abbreviate $x\cdot y$ by $xy$ and, in order to further simplify out notation, assume that $\cdot$ has priority over $\to$, which in turn has priority over $\meet$ and $\join$. For further information on commutative residuated lattices, see \cite{GJKO07,MPT23}.

We will adopt standard notation for general algebraic notations, referring, in particular, to the closure of a class $\K$ under isomorphic images, homomorphic images, subalgebras, direct products, and ultraproducts, respectively, by $\iso{(\K)}$, $\hm{(\K)}$, $\sub{(\K)}$, $\mathbb{P}{(\K)}$, and $\pu{(\K)}$. The variety generated by a class $\K$ of algebras is denoted by $\vr{(\K)}$. An algebra $\m{A}$ is said to be (finitely) subdirectly irreducible if whenever the least congruence $\Delta$ of $\m{A}$ is an intersection of a (finite) collection of congruences $\{\Theta_i\}_{i\in I}$, then $\Delta = \Theta_i$ for some $i\in I$. We denote the class of finitely subdirectly irreducible members of a class $\K$ by $\K_{\mathrm{FSI}}$.

Commutative residuated lattices comprise an arithmetical variety (i.e., congruence distributive and congruence permutable) with the congruence extension property, and same holds for commutative residuated lattices with designated bounds. Because all of these algebras are congruence distributive, J\'onsson's Lemma holds for them: If $\K$ is a class of commutative residuated lattices (with or without designated bounds), then the subdirectly irreducible members of $\vr(\K)$ are contained in $\hspu(\K)$. Additionally, if $\K$ is a class of totally ordered commutative residuated lattices, then the class of finitely subdirectly irreducibles members of $\vr{(\K)}$ is exactly the class of totally ordered members of $\vr{(\K)}$, which coincides with $\hspu(\K)$.

We will have occasion to refer to several special varieties of commutative residuated lattices (with or without designated bounds) in the sequel:
\begin{itemize}
\item Most importantly, a \emph{BL-algebra} is an integral commutative residuated lattice with designated bounds that satisfies the identities $x(x\to y) \approx x\meet y$ and $(x\to y)\join (y\to x) \approx 1$. We denote the variety of all BL-algebras by $\BL$.
\item An \emph{MV-algebra} is an integral commutative residuated lattice with designated bounds that satisfies the equation $(x\to y) \to y \approx x \join y$. We denote the variety of all MV-algebras by $\MV$. Note that every MV-algebra is a BL-algebra.
\item A \emph{basic hoop} is an integral commutative residuated lattice that satisfies $x(x\to y) \approx x\meet y$ and $(x\to y)\join (y\to x) \approx 1$. We denote the variety of all basic hoops by $\BH$.
\item A \emph{Wajsberg hoop} is an integral commutative residuated lattice that satisfies the equation $(x\to y) \to y \approx x \join y$. We denote the variety of all Wajsberg hoops by $\WH$.
\item A \emph{cancellative hoop} is a Wajsberg hoop that satisfies the $x\to (xy) \approx y$, or, equivalently, the quasiequation $xz\approx yz\Rightarrow x\approx y$.
\item An \emph{abelian $\ell$-group} is a commutative residuated lattice that satisfies the identity $x(x\to 1) \approx 1$, or, equivalently, a commutative residuated lattice where each element has an inverse.
\end{itemize}
All of the aforementioned varieties are \emph{semilinear} in the sense that they are generated by their totally ordered members. Every bounded Wajsberg hoop is the $0$-free reduct of an MV-algebra and, moreover, a Wajsberg hoop is either bounded or else is a cancellative hoop (see \cite{Agliano2002}). The only bounded cancellative hoop is the trivial one. Further, every basic hoop is a $0$-free subreduct of a BL-algebra.

For an abelian $\ell$-group $\m{G}$ with an element $u\geq 0$, we can define the MV-algebra $\m{\Gamma}( \m{G},u) = \alg{[0,u], \meet, \join,\cdot, \to ,u ,0} $ as follows:
\begin{itemize}
\item $[0,u] = \set{x \in G \mid 0\leq x \leq u}$,
\item the meet and join are just the restriction of the meet and join of $\m{G}$ to $[0,u]$,
\item for $a,b \in [0,u]$, $a\cdot b = (a+b - u) \join 0$,
\item for $a,b \in [0,u]$, $a \to b = (u + b -a) \meet u$.
\end{itemize}

Any totally ordered MV-algebra is isomorphic to an algebra of the form $\m{\Gamma}( \m{G},u) $ for a totally ordered abelian group $\m{G}$ with strong unit $u$. In particular any totally ordered bounded Wajsberg chain is isomorphic to the $0$-free reduct of such an algebra. See \cite{Cignoli2000} for details.

We will denote the standard MV-algebra on the unit interval by $[0,1]_{\MV}$, i.e., $[0,1]_{\MV} = \m{\Gamma}( \mathbb{R},1)$; the Wajsberg hoop comprising its $0$-free reduct will be denoted by $[0,1]_\WH$. For $m \in \mathbb{N}$ we define
\[
\Sm{m} = \m{\Gamma}(\mathbb{Z},m) \text{ and } \Smw{m} = \m{\Gamma}(\mathbb{Z}\times \mathbb{Z}, \pair{m,0}),
\]
where $\mathbb{Z}$ is ordered by the standard order and  $\mathbb{Z}\times \mathbb{Z}$ is lexicographically ordered from the left. We denote by $\Cn{m}$ and $\Cnw{m}$ the $0$-free reducts of $\Sm{m}$ and $\Smw{m}$, respectively. For both $\Sm{1}$ and $\Cn{1}$ we will also write $\m{2}$ when the context is clear. Note that $\Sm{m}$ is a homomorphic image of $\Smw{m}$, and hence $\Sm{m} \in \vr(\Smw{m})$. Similarly, $\Cn{m} \in \vr(\Cnw{m})$.

Given any commutative residuated lattice $\m{A} = \alg{A,\meet,\join,\cdot,\to,1}$, the \emph{negative elements} of $\m{A}$ are the members of the set
\[
A^- = \{a\in A \mid a\leq 1\}.
\]
The negative elements of any commutative residuated lattice are closed under the lattice operations and multiplication, and the \emph{negative cone} of $\m{A}$ is the residuated lattice $\m{A}^- = \{A^-,\meet,\join,\cdot,\to^-,1\}$, where $a\to^-b := (a\to b)\meet 1$. When $\m{A}$ is an abelian $\ell$-group, $\m{A}^-$ is a cancellative hoop. The negative cone of the integers $\mathbb{Z}$, considered as an abelian $\ell$-group, is especially important in the sequel. We denote this cancellative hoop by $\Cw$.

The lattices of subvarieties of $\MV$ and $\WH$ have been completely described in terms generating algebras of the constituent varieties, as summarized in the following propositions.

\begin{prop}[{\cite[Theorem 4.11]{Komori1981}}]\label{p:MV-char}
Let $\V$ be a proper subvariety of $\MV$. Then it is of one of the following two forms:
\begin{align*}
&\vr(\Sm{m_1},\dots, \Sm{m_r}) \quad (r\geq 1), \\
&\vr(\Sm{m_1},\dots, \Sm{m_r},\Smw{t_1},\dots, \Smw{t_s}) \quad (s\geq 1).
\end{align*}
\end{prop}

\begin{prop}[{\cite[Theorem 2.5]{Agliano2002}}]\label{p:WH-char}
Let $\cls{V}$ be a proper subvariety of $\WH$, then it is of one of the following three forms:
\begin{align*}
&\vr(\Cn{m_1},\dots, \Cn{m_r}) \quad (r\geq 1), \\
&\vr(\Cn{m_1},\dots, \Cn{m_r}, \Cw) \quad (r\geq 0), \\
&\vr(\Cn{m_1},\dots, \Cn{m_r},\Cnw{t_1},\dots, \Cnw{t_s}) \quad (s\geq 1).
\end{align*}
\end{prop}
In conjunction with the fact that $\MV = \vr{([0,1]_{\MV})}$ and $\WH = \vr{([0,1]_{\WH}})$, this shows that every subvariety of $\MV$ and $\WH$ is generated by finitely many totally ordered algebras.

For a bounded Wajsberg chain $\m{A}$, the \emph{radical} $\rad(\m{A})$ of $\m{A}$ is the intersection of all maximal, proper deductive filters of $\m{A}$ (see, e.g., \cite{Fussner2019,Agliano2002}). The radical of a bounded Wajsberg chain $\m{A}$ coincides with the set of all elements $a\in A$ such that $a^n\neq 0$ for any $n\in\mathbb{N}$. Note that $\rad(\m{A})$ is a filter and a cancellative subuniverse of $\m{A}$. We denote the corresponding algebra by $\Rad(\m{A})$. 

\begin{lemma}[See {\cite[Theorem 3.5.1]{Cignoli2000}}]\label{l:radical}
For a non-trivial bounded Wajsberg chain (totally ordered MV-algebra) $\m{A}$ the following are equivalent:
\begin{enumerate}[label = \textup{(\arabic*)}]
\item $\m{A}$ is simple.
\item $\rad(\m{A}) = \{1 \}$ .
\item $\m{A}$ is a subalgebra of the standard algebra $[0,1]_{\WH}$ ($[0,1]_\MV$).
\end{enumerate}
\end{lemma}

Apart from the ordinal sum, we will have need of another construction of totally ordered integral commutative residuated lattices, namely the \emph{disconnected rotation} (first introduced in \cite{Jenei2003}). Suppose that $\m{A} = \alg{A,\meet,\join,\cdot}$ is a totally ordered integral commutative residuated lattice. We define an algebra on $\rot{A} = (\{1\}\times A)\cup (\{0\}\times A)$ by setting
\[     (i,x)\join_\rot{A} (j,y) = \begin{cases}
        (1,x\join y) & \text{for } i=j=1\\
        (0,x\meet y) & \text{for } i=j=0\\
        (1,y) & \text{for } i<j
        \end{cases}
\]
\[     (i,x)\meet_\rot{A} (j,y) = \begin{cases}
        (1,x\meet y) & \text{for } i=j=1\\
        (0,x\join y) & \text{for } i=j=0\\
        (0,x) & \text{for } i<j
        \end{cases}
\]
\[     (i,x)\cdot_\rot{A} (j,y) = \begin{cases}
        (1,x\cdot y) & \text{for } i=j=1\\
        (0,1) & \text{for } i=j=0\\
        (0,y\to x) & \text{for } i<j
        \end{cases}
\]
\[     (i,x)\to_\rot{A} (j,y) = \begin{cases}
        (1,x\to y) & \text{for } i=j=1\\
        (1,y\to x) & \text{for } i=j=0\\
        (0,x\cdot y) & \text{for } j<i\\
        (1,1) & \text{for } i<j
        \end{cases}
\]
The algebra $\rot{\m{A}} = \alg{\rot{A},\meet_\rot{A},\join_\rot{A},\cdot_\rot{A},\to_\rot{A},(0,1),(1,1))}$ is called the \emph{disconnected rotation} of $\m{A}$.

\begin{lemma}[See \cite{AB21}]\label{l:rotation}
For any cancellative Wajsberg chain $\m{A}$:
\begin{enumerate}[label = \textup{(\roman*)}]
\item $\rot{\m{A}}$ is a totally ordered MV-algebra.
\item $\m{A}$ embeds into the $0$-free reduct of $\rot{\m{A}}$. 
\item For any totally ordered MV-algebra $\m{B}$, any embedding  $f\colon A \to B$ into the $0$-free reduct of $\m{B}$ extends to an embedding $\rot{f} \colon \rot{\m{A}} \to \m{B}$.
\end{enumerate}
\end{lemma}

\begin{lemma}[{\cite[Theorem 6.4]{Agliano2003}}]\label{l:ISPU-cancellative}
If $\m{A}$ is a non-trivial  cancellative Wajberg chain, then $\iso\sub\pu(\m{A}) = \iso\sub\pu(\Cw) = \hspu(\Cw)$.
\end{lemma}

\begin{cor}\label{c:ISPU-rotation}
If $\m{A}$ is a totally ordered MV-algebra with $\rad(\m{A}) \neq \{1\}$  and $\m{B}$ a totally ordered cancellative  hoop, then $\rot{\m{B}} \in \iso\sub\pu(\m{A})$. 
\end{cor}
\begin{proof}
Immediate from Lemma~\ref{l:rotation} and Lemma~\ref{l:ISPU-cancellative}.
\end{proof}

Now let $\alg{I,\leq}$ be a totally ordered set and $(\m{A}_i)_{i\in I}$ be an indexed family of totally ordered integral commutative residuated lattices with $A_i\cap A_j=\{1\}$ for $i\neq j$. We define the \emph{ordinal sum} $\bigoplus_{i\in I} \m{A}_i$ of the indexed family to be the algebra defined on $\m{A}=\bigcup_{i\in I} A_i$ such that:
\begin{itemize}
\item The order $\leq_{\m{A}}$ on $\m{A}$ extends the order on $\m{A}_i$ for each $i\in I$, and for $i<j$ and $x\in A_i\setminus \{1\}$, $y\in A_j\setminus\{1\}$, $x<y$.
\item  For each $x,y\in A$,
\[     x\cdot y = \begin{cases}
        x\cdot^{\m{A}_i} y & \text{for } x,y\in A_i\\
        y & \text{for } x\in A_i\text{ and }y\in A_j\setminus\{1\}\text{ with } j<i\\
        x & \text{for } x\in A_i\setminus\{1\}\text{ and }y\in A_j\text{ with } i<j
        \end{cases}
  \]
\item For each $x,y\in A$, 
\[ x\to y = \begin{cases}
        x\cdot^{\m{A}_i} y & \text{for } x,y\in A_i\\
        y & \text{for } x\in A_i\text{ and }y\in A_j\setminus\{1\}\text{ with } j<i\\
        1 & \text{for } x\in A_i\setminus\{1\}\text{ and }y\in A_j\text{ with } i<j
        \end{cases}
  \]
\end{itemize}

We will also consider ordinal sums $\m{A}=\bigoplus_{i\in I} A_i$, where $I$ has a least element $i_0$ and $\m{A}_{i_0}$ is non-trivial and has designated bounds. In this case, the ordinal sum is understood to have designated bounds with the least element from $\m{A}_{i_0}$. If $I = \{m,\dots, n\}$ for some $m,n\in \mathbb{N}$ and $m <n$ with the standard order, then we write $\bigoplus_{i=m}^n \m{A}_i$ for the ordinal sum $\bigoplus_{i\in I} \m{A}_i$ and we also write $\m{A}_1 \oplus \m{A}_2$ for binary ordinal sums, noting that the operation $\oplus$ is associative. Finally if $\K_i$ is a class of totally ordered integral commutative residuated lattices for each $i\in I$, then we denote by $\bigoplus_{i\in I} \K_i$ the class of ordinal sums $\bigoplus_{i\in I} \m{A}_i$ with $\m{A}_i \in \K_i$ for each $i\in I$.

\begin{lemma}[\cite{Agliano2003}]
\hfill
\begin{enumerate}[label = \textup{(\roman*)}]
\item For every totally ordered set  $\alg{I,\leq}$ and indexed family $\{\m{A}_i\}_{i\in I}$ of totally ordered basic hoops, the ordinal sum $\m{A}=\bigoplus_{i\in I} A_i$ is a totally ordered basic hoop.
\item For every totally ordered set $\alg{I,\leq}$ with minimum $i_0$,  totally ordered BL-algebra $\m{A}_{i_0}$, and indexed family $\{\m{A}_i\}_{i\in I{\setminus}\{i_0\}}$ of totally ordered basic hoops, the ordinal sum  $\bigoplus_{i\in I} \m{A}_i$ is a totally ordered BL-algebra.
\end{enumerate}
\end{lemma}

The following lemma summarizes the well-known decomposition of totally ordered BL-algebras and totally ordered basic hoops by the ordinal sum construction. It is fundamental to the work in subsequent sections.

\begin{lemma}[\cite{Agliano2003}]\label{lem:Agl}
\hfill
\begin{enumerate}[label = \textup{(\roman*)}]
\item For every totally ordered basic hoop $\m{A}$ there exists a unique (up to isomorphism) totally ordered set $\alg{I,\leq}$ and unique (up to isomorphism) Wajsberg chains $\m{A}_i$, $i\in I$, such that $\m{A} \cong \bigoplus_{i\in I} \m{A}_i$.

\item For every totally ordered BL-algebra $\m{A}$ there exists a unique (up to isomorphism) totally ordered set $\alg{I,\leq}$ with minimum $i_0$, a unique (up to isomorphism) totally ordered MV-algebra $\m{A}_{i_0}$, and  unique (up to isomorphism) Wajsberg chains $\m{A}_i$, $i\in I{\setminus}\{i_0\}$, such that $\m{A} \cong \bigoplus_{i\in I} \m{A}_i$.

\item For every totally ordered BL-algebra $\m{A}$ there exists a  a unique (up to isomorphism) totally ordered MV-algebra $\m{A}_0$ and a unique (up to isomorphism) totally ordered basic hoop $\m{A}_1$ such that $\m{A} \cong \m{A}_0 \oplus \m{A}_i$.
\end{enumerate}
\end{lemma}

It is implicit in the `uniqueness' part of the previous lemma that if ${\m A}$ is a totally ordered MV-algebra or Wajsberg hoop, then ${\m A}$ is \emph{sum-indecomposable} in the sense that if ${\m A}$ is not isomorphic to any non-trivial ordinal sum. Hence, totally ordered MV-algebras and Wajsberg hoops are sum-irreducible `building blocks' of totally ordered basic hoops and BL-algebras. The sum-irreducible components in ordinal sums may be detected in a quite tangible fashion, as attested in the following lemma.

\begin{lemma}[\cite{Agliano2003}]\label{l:component-equation}
Let $\m{A} = \bigoplus_{i\in I} \m{A}_i$  be a basic hoop such that $\m{A}_i$ is a Wajsberg chain for each $i\in I$. Then for all $a,b \in A$, $(a \to b) \to b = (b \to a ) \to a$ if and only if $a,b \in A_i$ for some $i\in I$.
\end{lemma}

If the unique-up-to-isomorphism totally ordered set $\alg{I,\leq}$ mentioned in Lemma~\ref{lem:Agl} is finite, we say that $\m{A}$ \emph{has finite index}. For any class $\K$ of basic hoops (or BL-algebras), let $\K_\mathrm{fc}$ be the class of totally ordered algebras in $\K$ with finite index. Note that $\K_\mathrm{fc}$ contains the class of finitely generated members of $\K$. It is well known that any variety is generated by its finitely generated members, so for any variety $\V$ of basic hoops (or BL-algebras) it holds that $\V=\vr(\Vfc)$. Moreover if $\W$ is another variety of basic hoops (or BL-algebras), then $(\V\join \W)_\mathrm{fc} = \Vfc \cup \Wfc$.

If we say that a totally ordered basic hoop (or BL-algebra) $\m{A} = \bigoplus_{i=0}^n \m{A}_i$ is of index $n+1$ then we always tacitly assume that each $\m{A}_i$ is a non-trivial totally ordered Wajsberg hoop (or MV-algebra).

The following technical lemmas provide needed information regarding the interaction between the class operators $\iso$, $\hm$, $\sub$, and $\pu$ and finite ordinal sums.

\begin{lemma}\label{l:BL-ordsum-embed}
\hfill
\begin{enumerate}[label = \textup{(\roman*)}]
\item  Let $\m{A} = \bigoplus_{i=1}^n \m{A}_i$ and $\m{B} = \bigoplus_{j=1}^m \m{B}_j$ be totally ordered basic hoops of index $n$ and $m$ respectively. A map $\phi \colon A  \to B$ is an embedding if and only if there exists an order-embedding $f\colon \{1,\dots, n\} \to \{1, \dots, m \}$ such that for each $i \in \{1,\dots, n \}$ the map $\phi$ restricts to an embedding from $\m{A}_i$ to $\m{B}_{f(i)}$. 

\item  Let $\m{A} = \bigoplus_{i=0}^n \m{A}_i$ and $\m{B} = \bigoplus_{j=0}^m \m{B}_j$ be totally ordered BL-algebras of index $n+1$ and $m+1$, respectively. A map $\phi \colon A  \to B$ is an embedding if and only if there exists an order-embedding $f\colon \{0,\dots, n\} \to \{0, \dots, m \}$, such that $f(0) = 0$ and for each $i \in \{0,\dots, n \}$ the map $\phi$ restricts to an embedding from $\m{A}_i$ to $\m{B}_{f(i)}$. 
\end{enumerate}
\end{lemma}

\begin{lemma}[\cite{Agliano2003}]\label{l:ISPU-ordinalsum}
Let $\m{A}$ be a totally ordered basic hoop (or BL-chain) with $\m{A} = \bigoplus_{i=0}^n \m{A}_i$. Then 
\begin{enumerate}[label = \textup{(\roman*)}]
\item $\iso\sub\pu(\bigoplus_{i=0}^n \m{A}_i) = \iso( \bigoplus_{i=0}^n \sub\pu( \m{A}_i))$.
\item $\hm(\bigoplus_{i=0}^n \m{A}_i) = \bigcup_{i=0}^n (( \bigoplus_{j=0}^{i-1} \m{A}_j) \oplus \hm(\m{A}_i))$.
\end{enumerate}
\end{lemma}

\begin{lemma}[\cite{Agliano2003}]\label{l:ultraprod-distrib-ordsum}
Let $X$ be a set, $n\in \mathbb{N}$, and for $k \in X$ let $\m{A}_0^k,\dots, \m{A}_n^k$ be totally ordered basic hoops. Then for every ultrafilter $U$ over $X$, 
\[
\prod_{k\in X} (\m{A}_0^k \oplus \dots \oplus \m{A}_n^k)/U \cong (\prod_{k\in X} \m{A}_0^k/U) \oplus \dots \oplus (\prod_{k\in X} \m{A}_n^k/U).
\]
\end{lemma}

\begin{lemma}\label{l:ordsum-component-var}
Let $\V$ be a variety of Wajsberg hoops and $\K = \{ \bigoplus_{i\in I} \m{A}_i \mid \m{A}_i \in \Vfsi \}$. Then $\hspu(\K) = \K$.
\end{lemma}

\begin{proof}
Let $\Sigma$ be an axiomatization of $\V$.  Define the term $c(x,y) = (x \to y) \to y$ and for each equation $\epsilon(x_1,\dots,x_n) \in \Sigma$ the quasiequation 
\[
\phi_\epsilon = \{ c(x_i,x_j) \approx c(x_i,x_j) \mid 1\leq i < j \leq n\} \Rightarrow \epsilon(x_1,\dots, x_n)
\]
and set $\Delta = \{\phi_\epsilon \mid \epsilon \in \Sigma \}$. Then it follows from Lemma~\ref{l:component-equation} that  $\K$ is axiomatized relative to the class of all totally ordered basic hoops by $\Delta$. Thus $\iso\sub\pu(\K) = \K$. Moreover, it is straightforward to check that also $\hm(\K) = \K$.
\end{proof}

\section{Essential extensions and amalgamation}\label{sec:essential}

\begin{figure}
\centering
\begin{tabular}{ccc}
\begin{tikzcd}
	& {\bf B} \\
	{\bf A} && {\bf D} \\
	& {\bf C}
	\arrow["{\phi_1}", hook, from=2-1, to=1-2]
	\arrow["{\psi_1}", dashed, hook, from=1-2, to=2-3]
	\arrow["{\phi_2}"', hook, from=2-1, to=3-2]
	\arrow["{\psi_2}"', dashed, hook, from=3-2, to=2-3]
\end{tikzcd}
& &
\begin{tikzcd}
	& {\bf B} \\
	{\bf A} && {\bf D} \\
	& {\bf C}
	\arrow["{\phi_1}", hook, from=2-1, to=1-2]
	\arrow["{\psi_1}", dashed, hook, from=1-2, to=2-3]
	\arrow["{\phi_2}"', hook, from=2-1, to=3-2]
	\arrow["{\psi_2}"', dashed,  from=3-2, to=2-3]
\end{tikzcd}
\\
(i) & & (ii)
\end{tabular}
\caption{Commutative diagrams for the amalgamation properties}
\label{fig:AP}
\end{figure}

Let $\K$ be any class of algebras. A \emph{span} in $\sf K$ is a pair of injective homomorphisms $\alg{\varphi_1\colon\mathbf{A}\to\m{B}, \phi_2\colon\mathbf{A}\to\mathbf{C}}$ where $\mathbf{A},\mathbf{B}, \mathbf{C} \in \mathsf{K}$. An \emph{amalgam} of a span $\alg{\varphi_1\colon\mathbf{A}\to\m{B}, \phi_2\colon\mathbf{A}\to\mathbf{C}}$ in $\sf K$ is a pair $\alg{\psi_1\colon\m{B}\to\m{D},\psi_2\colon\m{C}\to\m{D}}$ of embeddings where $\mathbf{D} \in \mathsf{K}$ and $\psi_1\circ \phi_1 = \psi_2 \circ \phi_2$, i.e., the diagram in Figure~\ref{fig:AP}(i) commutes. 

A \emph{one-sided amalgam} of a span $\alg{\varphi_1\colon\mathbf{A}\to\m{B}, \phi_2\colon\mathbf{A}\to\mathbf{C}}$ in $\sf K$ is a pair $\alg{\psi_1\colon\m{B}\to\m{D},\psi_2\colon\m{C}\to\m{D}}$ where $\mathbf{D} \in \mathsf{K}$, $\psi_1$ is an embedding, and $\psi_2$ is any homomorphism, such that $\psi_1\circ \phi_1 = \psi_2 \circ \phi_2$, i.e., the diagram in Figure~\ref{fig:AP}(ii) commutes. The class $\sf K$ has the \emph{amalgamation property} (or \emph{AP}) if every span in $\K$ has an amalgam in $\K$, and has the \emph{one-sided amalgamation property} (or \emph{one-sided AP}) if every span in $\K$ has a one-sided amalgam in $\K$.

\begin{lemma}[cf.~{\cite[Corollary~3.5]{Fussner2023}}]\label{lem:Fussner2023}
Let $\V$ be any variety with the congruence extension property such that $\sub(\Vfsi)=\Vfsi$. Then $\V$ has the amalgamation property if and only if (the class of finitely generated members of) $\Vfsi$ has the one-sided amalgamation property.
\end{lemma}

An extension $\m{A} \leq \m{B}$ is called \emph{essential}  if for every congruence $\theta \in \con(\m{B})$, whenever $\theta \neq \Delta_B$, also $\theta {\restriction}_A \neq \Delta_A$. An \emph{essential embedding} is an embedding $\phi\colon \m{A} \to \m{B}$ such that $\phi[\m{A}] \leq \m{B}$ is  essential. Note that an embedding $\phi \colon \m{A} \to \m{B}$ is essential if and only if for each homomorphism $\psi \colon \m{B} \to \m{C}$, if $\psi\circ \phi$ is injective, then $\psi$ is also injective.
Abusing the notation slightly we will also say that $\m{A} \leq \m{B}$ is an essential extension if there exists an essential embedding from $\m{A}$ into $\m{B}$.
We call a span $\pair{i_1 \colon \m{A} \to \m{B}, i_2\colon \m{A} \to \m{C}}$ \emph{essential} if $i_2$ is an essential embedding. We say that a class of algebras $\K$ has the \emph{essential amalgamation property} (or \emph{essential AP}) if every essential span in $\K$ has an amalgam in $\K$.

\begin{lemma}[\cite{Graetzer1971}]\label{l:essential-cong}
If $\m{A} \leq \m{B}$, then among the congruences $\theta \in \con(\m{B})$ with $\theta{\restriction}_A = \Delta_A$, there exists a maximal one, $\theta_0$, and the extension $\m{A} \leq \m{B}/\theta_0$ is essential.
\end{lemma}

\begin{lemma}\label{l:ess-AP}
 Let $\K$ be a class of algebras such that $\hm(\K) = \K$.
Then $\K$ has the one-sided amalgamation property if and only if $\K$ has the  essential amalgamation property.
\end{lemma} 
\begin{proof}
The left-to-right direction is immediate from the fact that clearly a one-sided amalgam of an essential span is an amalgam. For the right-to-left direction let $\pair{i_1 \colon \m{A} \to \m{B}, i_2\colon \m{A} \to \m{C}}$ be a span in $\K$. By Lemma~\ref{l:essential-cong}, there exists a congruence $\theta_0 \in \con(\m{C})$ with natural projection $\pi \colon \m{C} \to \m{C}/\theta_0$ such that $\pi\circ i_2$ is an essential embedding. Since $\hm(\K) = \K$, $\m{C}/\theta_0 \in \K$.  Hence $\pair{i_1 \colon \m{A} \to \m{B}, \pi\circ i_2\colon \m{A} \to \m{C}/\theta_0}$ is an essential span in $\K$ which, by assumption, has an amalgam $\pair{j_1 \colon \m{B} \to \m{D}, j_2 \colon \m{C}/\theta_0 \to \m{D}}$ in $\K$. But then clearly  $\pair{j_1 \colon \m{B} \to \m{D}, j_2 \circ \pi \colon \m{C}  \to \m{D}}$ is a one-sided amalgam of the original span.
\end{proof}

\begin{prop}\label{p:ess-AP}
Let $\V$ be a variety with the congruence extension property such that $\hm\sub(\Vfsi) = \Vfsi$. Then $\V$ has the amalgamation property if and only if $\Vfsi$ has the essential amalgamation property.

\end{prop}
\begin{proof}
By Lemma~\ref{lem:Fussner2023}, $\V$ has the AP if and only if $\Vfsi$ has the one-sided amalgamation property which, by Lemma~\ref{l:ess-AP}, is the case if and only if $\Vfsi$ has the essential AP.
\end{proof}

\begin{cor}\label{c:lin-ess-AP}
Let $\V$ be a variety of semilinear commutative (pointed) residuated lattices. Then $\V$ has the amalgamation property if and only if $\Vfsi$ has the essential amalgamation property.
\end{cor}

\begin{proof}
Note that $\Vfsi$ is the class of totally ordered members of $\V$. Hence it is clearly closed under $\hm\sub$. Moreover, any variety of commutative (pointed) residuated lattices has the congruence extension property. Thus the claim follows from Proposition~\ref{p:ess-AP}.
\end{proof}

\begin{cor}\label{c:Vfc-AP}
A variety $\V$ of basic hoops (or BL-algebras) has the amalgamation property if and only if $\Vfc$ has the essential amalgamation property.
\end{cor}

\begin{proof}
For the left-to-right direction suppose that $\V$ has the AP. Then, by Corollary~\ref{c:lin-ess-AP}, $\Vfsi$, has the essential AP. So every essential span in $\Vfc \subseteq \Vfsi$ has an amalgam in $\Vfsi$ and it is not hard to see that such an amalgam can restricted to an amalgam in $\Vfc$. 

For the right-to-left direction suppose that $\Vfc$ has the essential AP. Then, since $\hm\sub(\Vfc) =\Vfc$, by Lemma~\ref{l:ess-AP}, $\Vfc$ has the one-sided AP, i.e., every span in $\Vfc$ has an amalgam in $\Vfc$. So in particular every span of finitely generated members of $\Vfsi$ has a amalgam in $\V$ and, by Lemma~\ref{lem:Fussner2023}, $\V$ has the AP.
\end{proof}

Let $\K_1$ and $\K_2$ be classes of algebras similar algebras such that $\K_1\subseteq \K_2$. We say that $\K_1$ is \emph{essentially closed in} $\K_2$ if for any $\m{A} \in \K_1$ and essential embedding $\phi \colon \m{A} \to \m{B}$  with $\m{B} \in \K_2$, also $\m{B} \in \K_1$. 

\begin{remark}
Let $\K_1,\K_2,\K_3$ be classes of similar algebras.
\begin{enumerate}[label = \textup{(\roman*)}]
\item If $\K_1$ is essentially closed in $\K_2$ and $\K_3$, then $\K_1$ is essentially closed in $\K_2 \cup \K_3$.
\item If $\K_1$ and $\K_2$ are essentially closed in $\K_3$, then $\K_1 \cap \K_2$ and $\K_1 \cup \K_2$ are essentially closed in $\K_3$.
\item If $\K_1$ is essentially closed in $\K_2$ and $\K_2$ is essentially closed in  $\K_3$, then $\K_1$ is essentially closed in $\K_3$. 
\end{enumerate}
\end{remark}

\begin{lemma}\label{l:ess-AP-union}
Let $\K$ be a class of algebras such that $\K = \K_1 \cup \K_2$, $\iso\sub(\K_i) = \K_i$ for $i=1,2$, and $\K_1 \cap \K_2$  is essentially closed in $\K$. 
\begin{enumerate}[label = \textup{(\roman*)}]
\item If $\K_1$ and $\K_2$ have the essential amalgamation property, then $\K$ has the essential amalgamation property.
\item If $\K$ and $\K_1 \cap \K_2$ have the essential amalgamation property, then $\K_1$ and $\K_2$ have the essential amalgamation property. 
\end{enumerate}
\end{lemma}

\begin{proof}
(i) Let  $\pair{i_1 \colon \m{A} \to \m{B}, i_2\colon \m{A} \to \m{C}}$ be an essential span in $\K$.  If $\m{B}, \m{C} \in \K_i$ for $i=1,2$, then by assumption the span has an amalgam.  If $\m{B} \in \K_1$ and $\m{C} \in \K_2$, then $\m{A} \in \K_1 \cap \K_2$, since $\K_1$ and $\K_2$ are closed under $\iso\sub$. But, since $\K_1\cap \K_2$ is essentially closed in $\K$, we get $\m{C} \in \K_1\cap \K_2 \subseteq \K_1$ and, since $\K_1$ has the essential AP, the span has an amalgam in $\K_1$. Symmetrically it follows that if $\m{B} \in \K_2$ and $\m{C} \in \K_1$ then the span has an amalgam. 

(ii) Let $\pair{i_1 \colon \m{A} \to \m{B}, i_2\colon \m{A} \to \m{C}}$ be an essential span in $\K_1$. Then by assumption the span has an amalgam $\pair{j_1 \colon \m{B} \to \m{D}, j_2 \colon \m{C} \to \m{D}}$ in $\K$. If $\m{D} \in \K_2$, then $\m{A},\m{B },\m{C} \in \K_1 \cap \K_2$, since $\iso\sub(\K_2) = \K_2$. So, since $\K_1 \cap \K_2$ has the essential AP the span  has an amalgam  $\pair{j_1' \colon \m{B} \to \m{D}', j_2' \colon \m{C} \to \m{D}'}$ with  $ \m{D}' \in \K_1 \cap \K_2 \subseteq \K_1$. Symmetrically it follows that $\K_2$ has the essential AP.
\end{proof}

To apply the general results just announced, we will need some technical results regarding essential closedness and essential extensions in some relevant classes of algebras.

\begin{lemma}\label{l:rad-ess}
\hfill
\begin{enumerate}[label = \textup{(\roman*)}]
\item If $\m{A}$ is a non-trivial bounded Wajsberg chain with $\rad(\m{A}) \neq \{1\}$, then the inclusion homomorphism $\iota\colon \Rad(\m{A}) \hookrightarrow \m{A}$ is essential. 
\item If $\m{A}$ is a non-trivial totally ordered MV-algebra with $\rad(\m{A}) \neq \{1\}$, then the inclusion homomorphism $\rot{\iota}\colon \rot{\Rad(\m{A})} \hookrightarrow \m{A}$ is essential.
\end{enumerate}
\end{lemma}

\begin{proof}
(i) Since $\rad(\m{A})$ is a filter of $\m{A}$ and $\m{A}$ is totally ordered, it is clear that $\iota\colon \Rad(\m{A}) \hookrightarrow \m{A}$ is essential. Moreover, (ii) is immediate from (i).
\end{proof}

\begin{lemma}\label{l:bounded-ess-ext}
Let $\m{A}$ be a non-trivial bounded Wajsberg chain. Then every $\m{B} \in \hspu(\m{A})$ has a bounded essential extension in $\hspu(\m{A})$.
\end{lemma}
\begin{proof}
If $\m{B}$ is bounded, there is nothing to prove. If $\m{B}$ is cancellative, then, by \cite[Lemma 5]{AB21}, the $0$-free reduct of $\rot{\m{B}}$ is also contained in $\hspu(\m{A})$ which is clearly bounded and an essential extension of $\m{B}$.
\end{proof}

The following lemma is immediate from Lemma~\ref{l:ISPU-ordinalsum}(ii).

\begin{lemma}\label{l:ess-ordinalsum}
Let $\m{A} = \bigoplus_{i=0}^n \m{A}_i$ and $\m{B} = \bigoplus_{j=0}^m \m{B}_j$ be totally ordered basic hoops (or totally ordered BL-algebras) with $\m{A}_i, \m{B}_j$ non-trivial Wajsberg chains (or totally ordered MV-algebras). An embedding $\phi \colon \m{A} \to \m{B}$ is essential if and only if $\phi[\m{A}_n] \leq \m{B}_m$ is an essential extension.
\end{lemma}

Our study of the amalgamation property in varieties of BL-algebras and basic hoops will reduce to the study of amalgamation in associated classes of MV-algebras and Wajsberg hoops. For a class $\K$ of BL-algebras we denote by $\luk(\K)$ the class consisting of the totally ordered MV-algebras in $\hspu(\K)$ or, equivalently, the first components of ordinal sum decompositions of chains in $\hspu(\K)$. In particular, if $\V$ is a variety of BL-algebras, then $\luk(\V)$ is just the class of totally ordered MV-algebras in $\Vfc$.

\begin{lemma}\label{l:ess-closed-BL}
\hfill
\begin{enumerate}[label = \textup{(\roman*)}]
\item For any $n\in \mathbb{N}$, $\hspu(\Sm{n})$ (resp. $\hspu(\Cn{n})$) is essentially closed in $\hspu(\Smw{n})$ (resp. $\hspu(\Cnw{n})$).
\item For any variety $\V$ of BL-algebras  $\luk(\V)$ is  essentially closed in $\Vfsi$.
\item For any class $\K$ of BL-algebras such that $\hspu(\Sm{m}) \subseteq \K$ and $\luk(\K) = \hspu(\Smw{m})$ for some $m\in \mathbb{N}$, $\hspu(\Sm{m})$ is essentially closed in $\K$.
\end{enumerate}
\end{lemma}

\begin{proof}
(i) We prove the claim for $\Sm{n}$ the claim for $\Cn{n}$ then easily follows. First note that $\hspu(\Sm{n}) = \iso\sub(\Sm{n})$, so it contains only finite simple or trivial algebras. Hence, since essential extensions of simple totally ordered MV-algebras are simple because every non-trivial congruence identifies $0$ and $1$, it is enough to show that  $\hspu(\Smw{n})$ does not contain any simple algebra that is not contained in $\hspu(\Sm{n})$. Now, by \cite[Proposition 8.1.1]{Cignoli2000}, any infinite simple MV-algebra generates the variety $\MV$, so $\hspu(\Smw{n})$ contains only finite simple algebras. But, by \cite[Theorem 2.1]{Komori1981}, all of these finite simple algebras are already contained in $\hspu(\Sm{n})$. 

Part (ii) is  immediate from Lemma~\ref{l:ess-ordinalsum} and (iii) follows from (i) and (ii).

\end{proof}

Varieties of MV-algebras with the amalgamation property have been fully characterized by Di Nola and Lettieri, as summarized in the next proposition.

\begin{prop}[{\cite[Theorem 13]{DiNola2000}}]\label{p:AP-MV}
A variety of MV-algebras has the amalgamation property if and only if it is one-chain generated. 
\end{prop}

\begin{cor}\label{c:AP-MV}
A variety of MV-algebras has the amalgamation property if and only if it is of the form $\vr(\Sm{m})$ or $\vr(\Smw{m})$ for some $m\in \mathbb{N}$ or it is the variety $\MV = \vr([0,1]_{\MV})$.
\end{cor}

Varieties of Wajsberg hoops with the amalgamation property have also been fully described by Metcalfe, Montagna, and Tsinakis \cite{Metcalfe2014}. There, Wajsberg hoops are called \emph{commutative integral GMV-algebras}. The characterization of varieties of these with the AP is summarized as follows.

\begin{prop}[cf.~{\cite[Theorem 63]{Metcalfe2014}}]\label{p:AP-WH}
A variety of Wajsberg hoops has the amalgamation property if and only if it is of the form $\vr(\Cn{n})$, $\vr(\Cnw{n})$, $\vr(\Cw)$, or $\vr(\Cn{n}, \Cw)$ for some $n\in \mathbb{N}$, or it is the variety $\WH = \vr([0,1]_{\WH})$.
\end{prop}

From Lemma~\ref{lem:Fussner2023}, $\Vfsi$ has the one-sided AP for any variety $\V$ with the AP. For varieties of MV-algebras and Wajsberg hoops generated by a single chain, much more is true.

\begin{lemma}[{\cite[Theorem 5]{AB21}}] \label{l:one-chain}
\hfill
\begin{enumerate}[label = \textup{(\roman*)}]
\item For any one-chain generated variety $\V$ of MV-algebras the class  $\Vfsi$ of its totally ordered members has the amalgamation property.
\item For any one-chain generated variety of Wajsberg hoops the class of its totally ordered members has the amalgamation property.
\end{enumerate}
\end{lemma}

\section{Amalgamation in basic hoops}\label{sec:hoops}

Henceforth, for each variety $\V$ we will denote by $\Omega(\V)$ the collection of all subvarieties of $\V$ that have the AP. We will always consider $\Omega(\V)$ as a poset ordered by inclusion, i.e., as a subposet of the subvariety lattice of $\V$. Our aim in this section is to describe $\Omega(\BH)$. For a variety $\V$ of basic hoops, we denote by $\wajs(\V)$ the class of Wajsberg chains in $\V$. Note that $\wajs(\V) \subseteq \Vfc$.

\begin{lemma}
Let $\V \in \Omega(\BH)$. Then $\vr(\wajs(\V)) \in \Omega(\WH)$.
\end{lemma}

\begin{proof}
First note that, since $\V$ has the AP, by Corollary~\ref{c:Vfc-AP}, $\Vfc$ has the essential AP. Moreover, since $\wajs(\V) = \vr(\wajs(\V))_\mathrm{FSI}$, by Corollary~\ref{c:lin-ess-AP}, it is enough to show that $\wajs(\V)$ has the essential AP. Let $\pair{i_1 \colon \m{A} \to \m{B}, i_2 \colon \m{A} \to \m{C}}$ be an essential span in $\wajs(\V)$. If $\m{A}$ is trivial, then $\m{C}$ is trivial and clearly the span has an amalgam in $\wajs(\V)$. Otherwise, $\m{A}$ is non-trivial and it has an amalgam $\pair{j_1\colon \m{B} \to \m{D}, j_2 \colon \m{C} \to \m{D}}$ in $\Vfc$. Since $\m{A}$ is non-trivial, $\m{B}$ and $\m{C}$ are mapped into the same Wajsberg component $\m{D}' \in \wajs(\V)$ of $\m{D}$ yielding that $\pair{j_1\colon \m{B} \to \m{D}', j_2 \colon \m{C} \to \m{D}'}$ is an amalgam of the span in $\wajs(\V)$.
\end{proof}

So Proposition~\ref{p:AP-WH} yields the following result:
\begin{cor}\label{c:basic-wajscomp}
Let $\V \in \Omega(\BH)$ be non-trivial. Then  $\wajs(\V) = \hspu(\m{B})$ for some $\m{B} \in \{\Cn{n}, \Cnw{n} \mid n \in \mathbb{N}\} \cup \{ \Cw, [0,1]_\WH \}$ or $\wajs(\V) = \hspu(\Cn{n},\Cw)$ for some $n\geq 1$.
\end{cor}

In what follows, it will be convenient to use a more streamlined notation for the classes generating varieties of basic hoops (or BL-algebras). We introduce a number of conventions. First, we will sometimes abbreviate the ordinal  $\m{A}\oplus\m{B}$ by juxtaposition, writing $\m{A}\m{B}$ in its place. Second, we will denote the class generated by the componentwise $\hspu$ closure of an ordinal sum by enclosing the corresponding ordinal sum in bracket $[$, $]$, so that, for example, $[\m{A}\m{B}]$ denotes the class of all ordinal sums $\m{A}'\oplus\m{B}'$ where $\m{A}'\in\hspu(\m{A})$ and $\m{B}'\in\hspu(\m{B})$; and $[\m{A}] = \hspu(\m{A})$. In the event that the first component is a totally ordered MV-algebra, by convention we will only consider non-trivial members of the $\hspu$ closure; that is, we will consider the $\hm^+\sub\pu$ closure, where $\hm^+$ denotes the operator that closes under non-trivial homomorphic images. Thus, the resulting algebras will always be BL-algebras.
Thirdly, inspired by notation from regular expressions, we will use $^\ast$ to denote the repetition of one or more instances of a summand in a given ordinal sum, so that, for example, $[\m{A}\m{B}^\ast]$ abbreviates the class consisting of all ordinal sums of the form $\m{A}'\oplus\m{B}_1\oplus\cdots\oplus\m{B}_n$, where  $\m{A}' \in \hspu(\m{A})$, $n$ is a positive integer, and $\m{B}_1,\ldots,\m{B}_n\in\hspu(\m{B})$. We further adopt the convention that the Kleene star $^\ast$ has priority over $\oplus$, so that $[\m{A}\m{B}\m{C}^\ast]$ abbreviates $[(\m{A}\oplus\m{B})\oplus\m{C}^\ast]$. Thus, for instance, $[\m{A}(\Cn{n}\Cw)^\ast]$ is the class of all ordinal sums of the form $\m{A}'\oplus\m{B}_1\oplus\cdots\oplus\m{B}_n$, where  $\m{A}' \in \hspu(\m{A})$, $n$ is a positive integer, and $\m{B}_1,\m{B}_2,\ldots,\m{B}_n\in\hspu(\Cn{n}\Cw)$. Note that if $\m{A}$ and $\m{B}$ are Wajsberg chains, then $[(\m{A}\m{B})^\ast]$ is the class of all ordinal sums of the form $\m{C}_1\oplus \dots \oplus \m{C}_n$, where $n$ is a positive integer and $\m{C}_1,\dots, \m{C}_n \in \hspu(\m{A}, \m{B})$.

Recall that, for a variety $\V$ of basic hoops (or BL-algebras), $\Vfc$ is the class totally ordered members of $\V$ that have finite index and, by Corollary~\ref{c:Vfc-AP}, $\V\in\Omega(\BH)$ if and only if $\Vfc$ has the essential amalgamation property. The next lemma is important in describing the classes of finite index chains of varieties of basic hoops with the amalgamation property.

\begin{lemma}\label{l:basic-Vfc-regular-expr}
Let $\m{A}$ and $\m{B}$ be  non-trivial Wajsberg chains and 
\[
\K \in \{ [\m{A}], [\m{A}^\ast], [\m{A}\m{B}], [\m{A}^\ast\m{B}], [\m{A}\m{B}^\ast], [\m{A}^\ast\m{B}^\ast], [(\m{A}\m{B})^\ast] \}.
\] Then $\vr(\K)_{\mathrm{fc}} = \K$.
\end{lemma}
\begin{proof}
For $\K = [\m{A}]$ the claim is trivial, since $\hspu([\m{A}]) = [\m{A}]$, for $\K = [\m{A}^\ast]$ the claim follows from Lemma~\ref{l:ordsum-component-var}, and for $\K=  [\m{A}\m{B}]$ the claim follows from Lemma~\ref{l:ISPU-ordinalsum}. 
Moreover, for $\K = [(\m{A}\m{B})^\ast]$ we have $\K = \{\bigoplus_{i=1}^n \m{A}_i \mid \m{A}_i \in \hspu(\m{A},\m{B})\}$, so the claim also follows from Lemma~\ref{l:ordsum-component-var}.  
For the remaining cases we will only show the claim for $\K = [\m{A}^\ast \m{B}^\ast]$, since the other cases are very similar.  We define 
\begin{align*}
\K'_{\m{A}} &= \{( \bigoplus_{i\in I} \m{A}_i) \mid \m{A}_i \in \hspu(\m{A})\}, \\
\K'_{\m{B}} &= \{( \bigoplus_{j\in J} \m{B}_j) \mid \m{B}_j \in \hspu(\m{B})\}, \\
\K' &= \K'_{\m{A}} \oplus \K'_{\m{B}}.
\end{align*}
Note that $\K'_{\mathrm{fc}} = \K$ and it is enough to show that $\K'$ is closed under $\hspu$, since then $\K \subseteq \vr(\K)_\mathrm{fc} \subseteq \K'_\mathrm{fc} = \K$.  For closure under $\pu$ let $X$ be a set and $\{\m{A}^x \oplus \m{B}^x\}_{x \in X} \subseteq \K'$ with $\m{A}^x  \in \K'_{\m{A}}$ and $\m{B}^x \in \K'_{\m{B}}$. For any ultra filter $U$ over $X$ we have, by Lemma~\ref{l:ultraprod-distrib-ordsum}, $\prod_{x\in X} (\m{A}^x \oplus \m{B}^x)/U  \cong (\prod_{x\in X} \m{A}^x/U)\oplus (\prod_{x\in X} \m{B}^x/U)$. Moreover, by Lemma~\ref{l:ordsum-component-var}, we get $\prod_{x\in X} \m{A}^x/U \in \K'_{\m{A}}$ and  $\prod_{x\in X} \m{B}^x/U \in \K'_{\m{B}}$, yielding $\prod_{x\in X} (\m{A}^x \oplus \m{B}^x)/U \in \K'$. But also it is clear that $\sub(\K') = \K'$, so it remains to check that $\hm(\K') = \K'$, but this easily follows from Lemma~\ref{l:ISPU-ordinalsum}(ii) and the fact that $\K'_{\m{A}}$ and $\K'_{\m{B}}$ are closed under $\hm$, by Lemma~\ref{l:ordsum-component-var}.
\end{proof}

We now describe numerous families of varieties in $\Omega(\BH)$. The next several lemmas are devoted to this. Subsequently, we will show that every member of $\Omega(\BH)$ is contained in one of these families, and describe the structure of these families in more detail to give a complete description of $\Omega(\BH)$.

\begin{lemma}\label{l:AP-2comp}
Let $\m{A}$ be a non-trivial Wajsberg chain. Then the classes $[\m{A}]$ and $[\m{A}^\ast]$ have the amalgamation property.
\end{lemma}

\begin{proof}
Note that $[\m{A}]$ has the AP, by Lemma~\ref{l:one-chain}. To show that $[\m{A}^\ast]$ has the AP let $\pair{\phi_1 \colon \m{B} \to \m{C}, \phi_2 \colon \m{B} \to \m{D}}$ be a span in $[\m{A}^\ast]$. Then, by possibly introducing trivial algebras in the ordinal sum decomposition, we may assume that $\m{B} = \bigoplus_{i=1}^n \m{B}_i$, $\m{C} = \bigoplus_{i=1}^n \m{C}_i$, and $\m{D} = \bigoplus_{i=1}^n \m{D}_i$ with $\m{B}_i,\m{C}_i,\m{D}_i \in [\m{A}]$ and the span $\pair{\phi_1, \phi_2}$ restricts to spans $\pair{\phi_1^i \colon \m{B}_i \to \m{C}_i, \phi_2^i \colon \m{B}_i \to \m{D}_i}$ in $[\m{A}]$. Since $[\m{A}]$ has the AP, each of these spans has an amalgam $\pair{\psi_1^i \colon \m{C}_i \to \m{E}_i, \psi_1^i \colon \m{D}_i \to \m{E}_i}$ in $[\m{A}]$ giving rise to an amalgam $\pair{\psi_1 \colon \m{C} \to \m{E}, \psi_2 \colon \m{D} \to \m{E}}$ of $\pair{\phi_1,\phi_2}$ with $\m{E} = \bigoplus_{i=1}^n \m{E}_i \in [\m{A}^\ast]$. Thus $[\m{A}^\ast]$ has the AP.
\end{proof}

We call two algebras $\m{A}$ and $\m{B}$ of the same type \emph{independent} if $\vr(\m{A})$ and $\vr(\m{B})$ are independent, i.e., $\vr(\m{A}) \cap \vr(\m{B})$ only contains trivial algebras.

\begin{lemma}\label{l:AP-independent}
Let $\m{A}$ and $\m{B}$ be two independent Wajsberg chains. 
Then the classes $[\m{A}\m{B}]$, $[\m{A}  \m{B}^\ast]$, $[\m{A}^\ast\m{B}]$, $[ \m{A}^\ast  \m{B}^\ast]$,  and $[(\m{A} \m{B})^\ast]$ have the amalgamation property.
\end{lemma}

\begin{proof}
We prove that $[\m{A}^\ast\m{B}^\ast]$ has the AP. The proofs for the other cases are very similar. So let  $\pair{\phi_1 \colon \m{C} \to \m{D}, \phi_2 \colon \m{C} \to \m{E}}$ be a span in $[\m{A}^\ast\m{B}^\ast]$. Then, possibly with the introduction of trivial algebras in the ordinal sum decomposition, we may assume that $\m{C} = \m{C}_1 \oplus \m{C}_2$, $\m{D} = \m{D}_1 \oplus \m{D}_2$, $\m{E} = \m{E}_1 \oplus \m{E}_2$ with $\m{C}_1,\m{D}_1, \m{E}_1 \in [\m{A}^\ast]$ and $\m{C}_2,\m{D}_2, \m{E}_2 \in [\m{B}^\ast]$.  
Note that, since $\m{A}$ and $\m{B}$ are independent, also $[\m{A}^\ast]$ and $[\m{B}^\ast]$ are independent. Thus the span $\pair{\phi_1,\phi_2}$ restricts to spans $\pair{\phi_1^1 \colon \m{C}_1 \to \m{D}_1, \phi_2^1 \colon \m{C}_1 \to \m{E}_1}$ in $[\m{A}^\ast]$ and $\pair{\phi_1^2 \colon \m{C}_2 \to \m{D}_2, \phi_2^2 \colon \m{C}_2 \to \m{E}_2}$ in $[\m{B}^\ast]$. 
But, since $[\m{A}^\ast]$ and $[\m{B}^\ast]$ have the AP, by Lemma~\ref{l:AP-2comp}, these spans have amalgams $\pair{\psi_1^1 \colon \m{D}_1 \to \m{F}_1, \psi_2^1 \colon \m{E}_1 \to \m{F}_1}$ in $[\m{A}^\ast]$ and $\pair{\psi_1^2 \colon \m{D}_2 \to \m{F}_2, \psi_2^2 \colon \m{E}_2 \to \m{F}_2}$ in $[\m{B}^\ast]$, respectively. Hence the span $\pair{\phi_1,\phi_2}$ has an amalgam $\pair{\psi_1 \colon \m{D} \to \m{F}, \psi_2 \colon \m{E} \to \m{F}}$ in $[\m{A}^\ast\m{B}^\ast]$ with $\m{F} = \m{F}_1 \oplus \m{F}_2$.
\end{proof}

\begin{lemma}\label{l:AP-independent-cup}
Let $\m{A}$ and $\m{B}$ be two independent Wajsberg chains. 
Then the classes  $[\m{A}] \cup [\m{B}]$, $[\m{A}] \cup [\m{B}^\ast]$, and $[\m{A}^\ast] \cup [\m{B}^\ast]$ have the essential amalgamation property.
\end{lemma}

\begin{proof}
We prove that $[\m{A}^\ast] \cup [\m{B}^\ast]$ has the AP. The proofs for the other cases are very similar. Note that,  since $\m{A}$ and $\m{B}$ are independent, also $[\m{A}^\ast]$ and $[\m{B}^\ast]$ are independent. Thus $[\m{A}^\ast]\cap [\m{B}^\ast]$ contains only trivial algebras and in particular it is essentially closed in $[\m{A}^\ast]$ and $[\m{B}^\ast]$. Thus, since $[\m{A}^\ast]$ and $[\m{B}^\ast]$ have the AP, by Lemma~\ref{l:AP-2comp}, it follows from Lemma~\ref{l:ess-AP-union} that  $[\m{A}^\ast] \cup [\m{B}^\ast]$  has the essential AP.
\end{proof}

\begin{lemma}\label{l:AP-last}
Let $\m{A}$ and  $\m{B}$ be  Wajsberg chains such that $[\m{A}] \subseteq [\m{B}]$ is essentially closed in $[\m{B}]$. Then $[\m{A}^\ast  \m{B}]$ has the essential amalgamation property.
\end{lemma} 

\begin{proof}
Let $\pair{\phi_1 \colon \m{C} \to \m{D}, \phi_2 \colon \m{C} \to \m{E}}$ be an essential span in $[\m{A}^\ast\m{B}]$. 
Then, by Lemma~\ref{l:ess-ordinalsum}, $\m{C} = \m{C}_1 \oplus \m{C}_2$ and $\m{E} = \m{E}_1 \oplus \m{E}_2$ with $\m{C}_1, \m{E}_1 \in [\m{A}^\ast]$ and $\m{C}_2, \m{E}_2 \in [\m{B}]$ non-trivial such that $\phi_2$ restricts to embeddings $\phi_2^1 \colon \m{C}_1 \to \m{E}_1$, and $\phi_2^2 \colon \m{C}_2 \to \m{E}_2$ with $\phi_2^2$ essential. 

Suppose that  $\m{E}_2 \in [\m{B}]\setminus [\m{A}]$. Then, since $[\m{A}]$ is essentially closed in $[\m{B}]$, also $\m{C}_2 \in [\m{B}] \setminus [\m{A}]$, yielding that $\m{D}= \m{D}_1 \oplus \m{D}_2$ with $\m{D}_1 \in [\m{A}^\ast]$ and $\m{D}_2 \in [\m{B}] \setminus [\m{A}]$. Thus also $\phi_1$ restricts to embeddings $\phi_1^1 \colon \m{C}_1 \to \m{D}_1$, and $\phi_1^2 \colon \m{D}_2 \to \m{E}_2$. So $\pair{\phi_1,\phi_2}$ restricts to spans $\pair{\phi_1^1, \phi_2^1}$ in $[\m{A}^\ast]$ and $\pair{\phi_1^2,\phi_2^2}$ in $[\m{B}]$, which, by Lemma~\ref{l:AP-2comp}, have amalgams $\pair{\psi_1^1 \colon \m{D}_1 \to \m{F}_1, \psi_2^1 \colon \m{E}_1 \to \m{F}_1}$ in $[\m{A}^\ast]$ and $\pair{\psi_1^2 \colon \m{D}_2 \to \m{F}_2, \psi_2^2 \colon \m{E}_2 \to \m{F}_2}$ in $[\m{B}]$, respectively.
Hence the span $\pair{\phi_1,\phi_2}$ has an amalgam $\pair{\psi_1 \colon \m{D} \to \m{F}, \psi_2 \colon \m{E} \to \m{F}}$ in $[\m{A}^\ast\m{B}]$ with $\m{F} = \m{F}_1 \oplus \m{F}_2$.

Otherwise suppose that $\m{E}_2 \in [\m{A}]$, then also $\m{C}_2 \in [\m{A}]$. Thus, by possibly introducing extra trivial summands and redefining the summands $\m{C}_i, \m{E}_i$, we can assume that $\m{D} = \m{D}_1 \oplus \m{D}_2$ with $\m{D}_1 \in [\m{A}^\ast]$ and $\m{D}_2 \in [\m{B}]$ such that the span $\pair{\phi_1,\phi_2}$ restricts to spans $\pair{\phi_1^1 \colon \m{C}_1 \to \m{D}_1, \phi_2^1 \colon \m{C}_1 \to \m{E}_1}$ in $[\m{A}^\ast]$ and $\pair{\phi_1^2 \colon \m{C}_2 \to \m{D}_2, \phi_2^2 \colon \m{C}_2 \to \m{E}_2}$ in $[\m{B}]$. 
But, since $[\m{A}^\ast]$ and $[\m{B}]$ have the AP, by Lemma~\ref{l:AP-2comp}, these spans have amalgams $\pair{\psi_1^1 \colon \m{D}_1 \to \m{F}_1, \psi_2^1 \colon \m{E}_1 \to \m{F}_1}$ in $[\m{A}^\ast]$ and $\pair{\psi_1^2 \colon \m{D}_2 \to \m{F}_2, \psi_2^2 \colon \m{E}_2 \to \m{F}_2}$ in $[\m{B}]$, respectively. Hence the span $\pair{\phi_1,\phi_2}$ has an amalgam $\pair{\psi_1 \colon \m{D} \to \m{F}, \psi_2 \colon \m{E} \to \m{F}}$ in $[\m{A}^\ast\m{B}]$ with $\m{F} = \m{F}_1 \oplus \m{F}_2$.
\end{proof}

Having identified several varieties in $\Omega(\BH)$, we turn to proving that \emph{all} members of $\Omega(\BH)$ are among those just identified. For this, we will rely on several `closure properties', as exhibited, for example, in the next lemma.

\begin{lemma}\label{l:basic-ord-closure}
Let $\V \in \Omega(\BH)$.
\begin{enumerate}[label = \textup{(\roman*)}]
\item If $\bigoplus_{i=1}^l \m{A}_i \in \Vfc$ and for $1 \leq k \leq l$, $\m{A}_k \leq \m{B} \in \wajs(\V)$ is an essential extension, then $ (\bigoplus_{i=1}^{k-1} \m{A}_i) \oplus \m{B} \oplus (\bigoplus_{i=k+1}^l \m{A}_i) \in \Vfc$.
\item If $\bigoplus_{i=1}^l \m{A}_i \in \Vfc$, $\m{B} \in \wajs(\V)$, is simple, and for $1 \leq k \leq l$, $\m{A}_k \in [\m{B}]$ is non-trivial, then $ (\bigoplus_{i=1}^{k-1} \m{A}_i) \oplus \m{B} \oplus (\bigoplus_{i=k+1}^l \m{A}_i) \in \Vfc$.
\item If $\bigoplus_{i=1}^l \m{A}_i \in \Vfc$, $\Cnw{n}\in \Vfc$, and for $1 \leq k \leq l$, $\m{A}_k \in [\Cnw{n}]\setminus [\Cn{n}]$, then $ (\bigoplus_{i=1}^{k-1} \m{A}_i) \oplus \Cnw{n} \oplus (\bigoplus_{i=k+1}^l \m{A}_i) \in \Vfc$.
\item If $\bigoplus_{i=1}^l \m{A}_i, \bigoplus_{j=1}^m \m{B}_j \in \Vfc$ with $\m{A}_l$ non-trivial and $\m{A}_l \leq \m{B}_1 \in \wajs(\V)$, then $(\bigoplus_{i=1}^{l-1} \m{A}_i) \oplus (\bigoplus_{j=1}^m \m{B}_j) \in \Vfc$.
\end{enumerate}
\end{lemma}

\begin{proof}
Note first that, since $\V$ has the AP, by Corollary~\ref{c:Vfc-AP}, $\Vfc$ has the essential AP.

(i) If $\m{A}_k$ is trivial, then also $\m{B}$ is trivial and the claim is clear. Otherwise note that, by assumption, the span $\pair{\phi_1 \colon \m{A}_k \to \bigoplus_{i=1}^l \m{A}_i , \phi_2 \colon \m{A}_k \to \m{B}}$ with the obvious inclusion embeddings is essential. Hence it has an amalgam $\pair{\psi_1 \colon \bigoplus_{i=1}^l \m{A}_i \to \m{D}, \psi_2 \colon \m{B} \to \m{D}}$ in $\Vfc$. Since $\m{A}_k$ is non-trivial, $\psi_1(\m{A}_k)$ and $\psi_2(\m{B})$ are in the same Wajsberg component of $\m{D}$. Hence  $\m{D}$ contains an isomorphic copy of $(\bigoplus_{i=1}^{k-1} \m{A}_i) \oplus \m{B} \oplus (\bigoplus_{i=k+1}^l \m{A}_i)$ as a subalgebra and the claim follows.

(ii) If $\m{B}$ is cancellative, then, by Lemma~\ref{l:ISPU-cancellative}, $\iso\sub\pu(\m{A}_k) = \iso\sub\pu(\m{B})$ and the claim follows from Lemma~\ref{l:ISPU-ordinalsum}(i). Otherwise $\m{B}$ is bounded, so $\m{A}_k$ has a bounded essential extension $\m{C} \in [\m{B}] \subseteq \wajs(\V)$, by Lemma~\ref{l:bounded-ess-ext}. Thus, by part (i), $(\bigoplus_{i=1}^{k-1} \m{A}_i) \oplus \m{C} \oplus (\bigoplus_{i=k+1}^l \m{A}_i) \in \Vfc$. But also, since $\m{2} \leq \m{C}$, $(\bigoplus_{i=1}^{k-1} \m{A}_i) \oplus \m{2} \oplus (\bigoplus_{i=k+1}^l \m{A}_i) \in \Vfc$, yielding $(\bigoplus_{i=1}^{k-1} \m{A}_i) \oplus \m{B} \oplus (\bigoplus_{i=k+1}^l \m{A}_i) \in \Vfc$, by part (i), because $\m{2} \leq \m{B}$ is essential as $\m{B}$ is simple.

(iii) Note that, since $\m{A}_k \in [\Cnw{n}]\setminus [\Cn{n}]$, by \cite[Proposition 2.8]{Agliano2002}, $\m{A}_k$ is infinite, i.e., either $\m{A}_k$ is non-trivial and cancellative or $\m{A}_k$ is bounded and $\rad(\m{A}) \neq \{1\}$. Thus, by Lemma~\ref{l:ISPU-cancellative}, $\Rad(\Cnw{n}) \in \iso\sub\pu(\m{A}_k)$ and Lemma~\ref{l:ISPU-ordinalsum}(i) yields $(\bigoplus_{i=1}^{k-1} \m{A}_i) \oplus \Rad(\Cnw{n}) \oplus (\bigoplus_{i=k+1}^l \m{A}_i) \in \Vfc$. Hence, since $ \Rad(\Cnw{n})\leq \Cnw{n}$ is essential, by part (i), $(\bigoplus_{i=1}^{k-1} \m{A}_i) \oplus \Cnw{n} \oplus (\bigoplus_{i=k+1}^l \m{A}_i) \in \Vfc$.

(iv) Note first that the span $\pair{\phi_1 \colon \m{A}_l \to \bigoplus_{j=1}^m \m{B}_j, \phi_2 \colon \m{A}_l \to \bigoplus_{i=1}^l \m{A}_i}$ with the obvious inclusion embeddings is essential, by Lemma~\ref{l:ess-ordinalsum},   since $\m{A}_k$ is non-trivial. Thus it has an amalgam $\pair{\psi_1 \colon \bigoplus_{j=1}^m \m{B}_j \to \m{D}, \psi_2 \colon \bigoplus_{i=1}^l \m{A}_i \to \m{D}}$. Moreover, since $\m{A}_l$ is non-trivial, $\psi_1(\m{B}_1)$ and $\psi_2(\m{A}_l)$ are in the same Wajsberg component. Hence $(\bigoplus_{i=1}^{l-1} \m{A}_i) \oplus (\bigoplus_{j=1}^m \m{B}_j)$ is isomorphic to a subalgebra of $\m{D}$ an the claim follows.
\end{proof}

We will further establish several more closure properties that will be used in our characterization of varieties of basic hoops with the amalgamation property.

\begin{lemma}\label{l:basic-pw-hom}
Let $\V \in \Omega(\BH)$. If $\bigoplus_{i=1}^l \m{A}_i \in \Vfc$ with $\m{A}_i \in \wajs(\V)$, then $\bigoplus_{i=1}^l \hm(\m{A}_i) \subseteq \Vfc$
\end{lemma}

\begin{proof}
It is enough to show for any  $k \in \{1,\dots, l\}$, that for any homomorphic image $\m{B}$ of $\m{A}_k$, $ (\bigoplus_{i=1}^{k-1} \m{A}_i) \oplus \m{B} \oplus (\bigoplus_{i=k+1}^l \m{A}_i) \in \Vfc$. If $\m{B}$ is trivial, then clearly $ (\bigoplus_{i=1}^{k-1} \m{A}_i) \oplus \m{B} \oplus (\bigoplus_{i=k+1}^l \m{A}_i) =  (\bigoplus_{i=1}^{k-1} \m{A}_i) \oplus (\bigoplus_{i=k+1}^l \m{A}_i)  \in \Vfc$. Otherwise $\m{B}$ is non-trivial.
If $\m{B}$ is cancellative, then also $\m{A}_k$ is cancellative and the claim follows from Lemma~\ref{l:ISPU-cancellative} and Lemma~\ref{l:ISPU-ordinalsum}(i). So we may assume that $\m{B}$ and $\m{A}_k$ are bounded. Thus, by Lemma~\ref{l:radical}, either $\m{B}$ is simple or $\rad(\m{B}) \neq \{1 \}$.

If $\m{B}$ is simple, then, since $\m{A}_k$ is bounded and $\m{2} \leq \m{A}_k$,  $(\bigoplus_{i=1}^{k-1} \m{A}_i) \oplus \m{2} \oplus (\bigoplus_{i=k+1}^l \m{A}_i) \in \Vfc$. But also $\m{B}$ is simple and  $\m{2} \in  [\m{B}]$, so Lemma~\ref{l:basic-ord-closure}(ii) yields that  $ (\bigoplus_{i=1}^{k-1} \m{A}_i) \oplus \m{B} \oplus (\bigoplus_{i=k+1}^l \m{A}_i) \in \Vfc$. 

Otherwise  $\rad(\m{B}) \neq \{1 \}$. Hence also $\rad(\m{A}_k) \neq \{1 \}$ and, by Lemma~\ref{l:ISPU-cancellative}, $\Rad(\m{B}) \in \iso\sub\pu(\m{A}_k)$. So, by Lemma~\ref{l:ISPU-ordinalsum}(i), we obtain $ (\bigoplus_{i=1}^{k-1} \m{A}_i) \oplus \Rad(\m{B}) \oplus (\bigoplus_{i=k+1}^l \m{A}_i) \in \Vfc$. Moreover, by Lemma~\ref{l:rad-ess}, $\Rad(\m{B}) \leq \m{B}$ is essential. Hence, Lemma~\ref{l:basic-ord-closure}(i) yields that  $ (\bigoplus_{i=1}^{k-1} \m{A}_i) \oplus \m{B} \oplus (\bigoplus_{i=k+1}^l \m{A}_i) \in \Vfc$. 
\end{proof}

Lemma~\ref{l:ISPU-ordinalsum} and Lemma~\ref{l:basic-pw-hom} yield:

\begin{prop}\label{p:basic-pw-hspu}
Let $\V \in \Omega(\BH)$. If $\bigoplus_{i=1}^l \m{A}_i \in \Vfc$, then $[\m{A}_1 \cdots \m{A}_l] \subseteq \Vfc$.
\end{prop}

\begin{lemma}\label{l:basic-properties-classes}
Let $\V \in \Omega(\BH)$ and $\m{A},\m{B},\m{C}$ Wajsberg chains.
\begin{enumerate}[label = \textup{(\arabic*)}]
\item If $\m{A} \oplus \m{B} \oplus \m{B} \oplus \m{C} \in \Vfc$, then $[\m{A}\m{B}^\ast\m{C}] \subseteq \Vfc$.
\item If $\Cn{n}\oplus \Cnw{n} \in \Vfc$, then $[\Cn{n}^\ast\Cnw{n}]\subseteq \Vfc$.
\item  If $\m{A} \oplus \m{A} \oplus \m{B} \oplus \m{B} \in \Vfc$ or $\m{A} \oplus \m{A}, \m{A} \oplus \m{B}, \m{B} \oplus \m{B} \in \Vfc$, then $[\m{A}^\ast \m{B}^\ast] \subseteq \Vfc$.
\item If $\m{A} \oplus \m{B}, \m{B} \oplus \m{A} \in \Vfc$, then $[(\m{A}\m{B})^\ast] \subseteq \Vfc$.
\end{enumerate}
\end{lemma}

\begin{proof}
(1) If $\m{B}$ is trivial, then the claim is immediate from Proposition~\ref{p:basic-pw-hspu}. Otherwise $\m{B}$ is non-trivial and, by Proposition~\ref{p:basic-pw-hspu}, it is enough to show that $\m{A} \oplus (\bigoplus_{i=1}^k \m{B}) \oplus \m{C} \in \Vfc$ for every $k\geq 1$. We show the claim by induction on $k$. The base case $k = 1$ is clear since $\m{A} \oplus \m{B} \oplus \m{C}  \leq \m{A} \oplus \m{B} \oplus \m{B} \oplus \m{C}$. Suppose that  $\m{A} \oplus (\bigoplus_{i=1}^k \m{B}) \oplus \m{C} \in \Vfc$ for $k\geq 1$. Then also $\m{A} \oplus (\bigoplus_{i=1}^k \m{B}) \in \Vfc$ and $\m{B} \oplus \m{B} \oplus \m{C} \in \Vfc$. Thus, by Lemma~\ref{l:basic-ord-closure}(iv), $\m{A} \oplus (\bigoplus_{i=1}^{k-1} \m{B}) \oplus  \m{B} \oplus \m{B} \oplus \m{C} \in \Vfc$.

(2) Note that, since $\Cn{n}\oplus \Cnw{n}  \in \Vfc$, also $\Cn{n}\oplus \Cn{n}  \in \Vfc$. So, by Lemma~\ref{l:basic-ord-closure}(iv), $\Cn{n}\oplus \Cn{n} \oplus  \Cnw{n}  \in \Vfc$ and the claim follows  from (1).

The proofs of (3) and (4) are very similar to the proof of (1).
\end{proof}

We define the collection  
\[
\BHAP = \{ \Vfc \mid \V\in\Omega(\BH)\},
\]
and we will consider it being ordered by $\subseteq$. Note that for varieties of basic hoops $\V$ and $\W$, $\V \subseteq \W$ iff $\Vfc \subseteq \Wfc$. Thus $\alg{\BHAP,\subseteq}$ is isomorphic to the poset $\Omega(\BH)$ of varieties of basic hoops with the amalgamation property.

\begin{prop}\label{p:intervals-basic}
The poset $\alg{\BHAP, \subseteq}$ can be partitioned into countably infinitely many closed intervals: for any variety $\V$ of basic hoops with the amalgamation property one of the following holds:
\begin{enumerate}[label = \textup{(\arabic*)}]
\item $\V$ is trivial.
\item $\wajs(\V) = [\m{A}]$ for $\m{A} \in \{\Cn{n} \mid n \geq 1\} \cup \{\Cw, [0,1]_\WH\}$, and $[\m{A}] \subseteq \Vfc \subseteq [\m{A}^\ast]$.
\item $\wajs(\V) = [\Cnw{n}]$ for some $n\geq 1$, and $[\Cnw{n}] \subseteq \Vfc \subseteq [\Cnw{n}^\ast]$.
\item $\wajs(\V) = [\Cn{n}] \cup [\Cw]$ for some $n\geq 1$, and $[\Cn{n}] \cup [\Cw] \subseteq \Vfc \subseteq [(\Cn{n}\Cw)^\ast]$.
\end{enumerate}
\end{prop}

\begin{proof}
First note that, by Corollary~\ref{c:basic-wajscomp}, that either $\V$ is trivial,  $\wajs(\V) = [\m{A}]$ for $\{\Cn{n},\Cnw{n} \mid n \geq 1\} \cup \{\Cw, [0,1]_\WH\}$, or $\wajs(\V) = [\Cn{n}] \cup [\Cw]$ for some $n\geq 1$. 

(2) and (3): If $\wajs(\V) = [\m{A}]$ for a non-trivial Wajsberg chain $\m{A}$, then clearly $[\m{A}] \subseteq \Vfc \subseteq [\m{A}^\ast]$ and, by Lemma~\ref{l:basic-Vfc-regular-expr} and Lemma~\ref{l:AP-2comp}, $[\m{A}], [\m{A}^\ast] \in \BHAP$.

(4) If  $\wajs(\V) = [\Cn{n}] \cup [\Cw]$ for some $n\geq 1$, then clearly $[\Cn{n}] \cup [\Cw] \subseteq \Vfc \subseteq [(\Cn{n}\Cw)^\ast]$. 
Moreover, by Lemma~\ref{l:AP-independent-cup}, Lemma~\ref{l:AP-independent}, and Lemma~\ref{l:basic-Vfc-regular-expr}, $[\Cn{n}] \cup [\Cw], [(\Cn{n}\Cw)^\ast]  \in \BHAP$.
\end{proof}

We denote the intervals corresponding to the cases (2), (3), and (4) by $\bvarint{\m{A}}$, $\bvarint{\Cnw{n}}$, and $\bvarint{\Cn{n},\Cw}$, respectively. We will show that each of these intervals is finite, and give explicit descriptions of them.

\begin{figure}
\centering
\small
\begin{tabular}{ccc}

\begin{tikzpicture}
\node (1) at (0,0) {$[\m{A}]$};
\node (2) at (0,1) {$[\m{A}^\ast]$};
\draw (1) -- (2);
\end{tikzpicture}
&
\quad\quad\quad
&
\begin{tikzpicture}
\node (1) at (0,0) {$[\Cnw{n}]$};
\node (2) at (0,1) {$[\Cn{n}^\ast\Cnw{n}]$};
\node (3) at (0,2) {$[\Cnw{n}^\ast]$};
\draw (1) -- (2) -- (3);
\end{tikzpicture}
\\
(i) & &(ii) 
\end{tabular}
\caption{The intervals $\bvarint{\m{A}}$ and $\bvarint{\Cnw{n}}$}
\label{fig:basic1}
\end{figure}

\begin{lemma}\label{l:AP-basic-inteval1}
Let $\m{A} \in \{\Cn{n} \mid n \geq 1\} \cup \{\Cw, [0,1]_\WH\}$. Then the order of the interval $\bvarint{\m{A}}$ is as depicted in Figure~\ref{fig:basic1}(i).
\end{lemma}

\begin{proof}
Suppose that $\Vfc \in \bvarint{\m{A}}$ and $[\m{A}]\subsetneq \Vfc$. 
Then, since $[\m{A}]$ contains all the members of $[\m{A}^\ast]$ with index $1$, there is a $\m{B} \in \Vfc$ such that $\m{B} = \m{B}_1 \oplus \m{B}_2$ and $\m{B}_1,\m{B}_2 \in [\m{A}]$ are non-trivial. But then, since $\m{A}$ is simple, by Lemma~\ref{l:basic-ord-closure}(ii), $\m{A}\oplus \m{A} \in \Vfc$. Hence Lemma~\ref{l:basic-properties-classes}(1) yields $[\m{A}^\ast]\subseteq \Vfc$.
\end{proof}

\begin{lemma}\label{l:AP-basic-interval2}
Let $n\geq 1$. Then the order of the interval $\bvarint{\Cnw{n}}$ is as depicted in Figure~\ref{fig:basic1}(ii).
\end{lemma}

\begin{proof}
First note that, by Lemma~\ref{l:ess-closed-BL}(i), $[\Cn{n}]$ is essentially closed in $[\Cnw{n}]$. Thus, by Lemma~\ref{l:AP-last} and Lemma~\ref{l:basic-Vfc-regular-expr}, $[\Cn{n}^\ast\Cnw{n}] \in \BHAP$,

Suppose that $\Vfc \in \bvarint{\Cnw{n}}$ and $[\Cnw{n}]\subsetneq \Vfc$.  Then, since $[\Cnw{n}]$ contains all the members of $[\Cnw{n}^\ast]$ with index $1$, there is a $\m{B} \in \Vfc$ such that $\m{B} = \m{B}_1 \oplus \m{B}_2$ and $\m{B}_1,\m{B}_2 \in [\Cnw{n}]$ are non-trivial. But then, by Lemma~\ref{l:bounded-ess-ext}, $\m{B}_1$ and $\m{B}_2$ have bounded essential extensions $\m{C}_1$ and $\m{C}_2$ in $[\Cnw{n}]$, respectively. So Lemma~\ref{l:basic-ord-closure}(i) yields $\m{C}_1 \oplus \m{C}_2 \in \Vfc$, and, since $\m{C}_1$ and $\m{C}_2$ are bounded, $\m{2} \oplus \m{2} \in \Vfc$. Thus, by Lemma~\ref{l:basic-ord-closure}(ii), $\Cn{n} \oplus \m{2} \in \Vfc$, and by Lemma~\ref{l:basic-ord-closure}(iv), $\Cn{n} \oplus \Cnw{n} \in  \Vfc$. Hence  Lemma~\ref{l:basic-properties-classes}(2) yields $[\Cn{n}^\ast\Cnw{n}] \subseteq \Vfc$.

Suppose that $\Vfc \in \bvarint{\Cnw{n}}$ and $[\Cn{n}^\ast\Cnw{n}]\subsetneq \Vfc$.  Then, since $[\Cn{n}^\ast\Cnw{n}]$ contains all the members of $[\Cnw{n}^\ast]$ with all but possibly the last Wajsberg component in $[\Cn{n}]$, there is a $\m{B} \in \Vfc$ such that $\m{B} = \m{B}_1 \oplus \m{B}_2$ and $\m{B}_1 \in [\Cnw{n}]\setminus [\Cn{n}]$ and $\m{B}_2 \in [\Cnw{n}]$ non-trivial. But then, by Lemma~\ref{l:basic-ord-closure}(iii), $\Cnw{n}\oplus \m{B}_1 \in \Vfc$ and arguing as above also $\Cnw{n}\oplus \Cnw{n} \in \Vfc$. Hence  Lemma~\ref{l:basic-properties-classes}(2) yields $[\Cnw{n}^\ast] \subseteq \Vfc$.
\end{proof}

\begin{figure}
\small
\centering
\begin{tikzpicture}
\node (0) at (0,0) {\wuz};
\node (1) at (-2,2) {\wsuz};
\node (2) at (2,2) {\wuzs};
\node (3) at (5,2) {\wz};
\node (4) at (-5,2) {\zw};
\node (5) at (0,4) {\wsuzs};
\node (6) at (3,4) {\wsz};
\node (7) at (5,4) {\wzs};
\node (8) at (-3,4) {\zsw};
\node (9) at (-5,4) {\zws};
\node (10) at (3,6) {\wszs};
\node (11) at (-3,6) {\zsws};
\node (12) at (0,8) {\wzd};

\draw (0) -- (1) -- (5) -- (2) -- (0);
\draw (0) -- (4) -- (8) -- (2) -- (7);
\draw (0) -- (3) -- (6) -- (1) -- (9);
\draw (4) -- (9) -- (11) -- (8);
\draw (3) -- (7) -- (10) -- (6);
\draw (5) -- (10) -- (12) -- (11) -- (5);
\end{tikzpicture}
\caption{The interval $\bvarint{\{ \Cn{n},\Cw \}}$}
\label{fig:basic2}
\end{figure}

\begin{lemma}\label{l:AP-basic-interval3}
Let $n\geq1$. Then the order of the interval $\bvarint{\Cn{n}, \Cw}$ is as depicted in Figure~\ref{fig:basic2}.
\end{lemma}

\begin{proof}
First note that, by Lemma~\ref{l:AP-2comp}, Lemma~\ref{l:AP-independent}, and Lemma~\ref{l:AP-independent-cup} together with Lemma~\ref{l:basic-Vfc-regular-expr}, all the classes in the figure are contained in $\BHAP$ and thus in particular in $\bvarint{\Cn{n}, \Cw}$. Moreover, they are clearly ordered as depicted in the figure. We prove the claim by establishing the covers of all the classes in $\bvarint{\Cn{n}, \Cw}$ starting from the least element $\mwuz$.

\begin{enumerate}[label = (\alph*)]
\item Suppose that $\Vfc \in \bvarint{\Cn{n}, \Cw}$ and $\mwuz \subsetneq \Vfc$. Then, since $\mwuz$ contains all the members of $\mwzd$ with index $1$, there exists a $\m{A} \in \Vfc$ with $\m{A} = \m{A}_1 \oplus \m{A}_2$ such that $\m{A}_1,\m{A}_2 \in \mwuz$ are non-trivial. There are four cases. 

If $\m{A}_1, \m{A}_2 \in [\Cn{n}]$, then, by Lemma~\ref{l:basic-ord-closure}(ii), $\Cn{n}\oplus \Cn{n} \in \Vfc$, so, by Lemma~\ref{l:basic-properties-classes}(ii), $\mwsuz \subseteq \Vfc$. 

Similarly,  if $\m{A}_1, \m{A}_2 \in [\Cw]$, then $\mwuzs \subseteq \Vfc$.   

If $\m{A}_1 \in [\Cn{n}]$ and $\m{A}_2 \in [\Cw]$, then, by Lemma~\ref{l:basic-ord-closure}(ii), $\Cn{n}\oplus \Cw \in \Vfc$, so, by Lemma~\ref{l:basic-pw-hom}, $\mwz \subseteq \Vfc$. 

Similarly, if $\m{A}_1 \in [\Cw]$ and $\m{A}_2 \in [\Cn{n}]$, then $\mzw \subseteq \Vfc$.

\item Suppose that $\Vfc \in \bvarint{\Cn{n}, \Cw}$ and $\mwz \subsetneq \Vfc$. Then, since  $\mwz$ contains all members of $\mwzd$ of index $1$ and of index $2$ with first component in $[\Cn{n}]$ and second component in $[\Cw]$, there are three cases: either $\m{A}_1 \oplus \m{A}_2 \in \Vfc$ with $\m{A}_1 \in [\Cw]$ and $\m{A}_2 \in [\Cn{n}]$ non-trivial; or $\m{B}_1 \oplus \m{B}_2 \oplus \m{B}_3 \in \Vfc$ with   $\m{B}_1,\m{B}_2 \in [\Cn{n}]$ and $\m{B}_3 \in [\Cw]$ non-trivial; or $\m{C}_1 \oplus \m{C}_2 \oplus \m{C}_3 \in \Vfc$ with   $\m{C}_1\in [\Cn{n}]$ and $\m{C}_2 ,\m{C}_3 \in [\Cw]$ non-trivial. 

In the first case it follows from Lemma~\ref{l:basic-ord-closure}(ii) that $\Cw\oplus \Cn{n} \in \Vfc$ and, since $\Cn{n}\oplus \Cw\in \Vfc$, Lemma~\ref{l:basic-properties-classes}(4) yields $\mwzd \subseteq \Vfc$. So, in particular, $\mwsz \subseteq \Vfc$ and $\mwzs \subseteq \Vfc$. 

In the second case it follows from Lemma~\ref{l:basic-ord-closure}(ii) that $\Cn{n}\oplus \Cn{n} \oplus \Cw \in \Vfc$. Thus, by Lemma~\ref{l:basic-properties-classes}(1), $\mwsz \subseteq \Vfc$. 
Similarly it follows in the third case that $\mwzs \subseteq \Vfc$.

\item Suppose that $\Vfc \in \bvarint{\Cn{n}, \Cw}$ and $\mzw \subsetneq \Vfc$. Then it follows similarly as in (b) that either $\mzws \subseteq \Vfc$ or $\mzsw \subseteq \Vfc$.

\item Suppose that  $\Vfc \in \bvarint{\Cn{n}, \Cw}$ and $\mwuzs \subsetneq \Vfc$. Then, since $\mwuzs$ contains all members of $\mwzd$ with index $1$ or with all components in $[\Cw]$, there is a $\m{A} \in \Vfc$ of index $2$ that has a component in $[\Cn{n}]$. There are three cases:

If $\m{A} = \m{A}_1 \oplus \m{A}_2$ with $\m{A}_1,\m{A}_2 \in [\Cn{n}]$ non-trivial, then, by Lemma~\ref{l:basic-ord-closure}(ii), $\Cn{n}\oplus \Cn{n} \in \Vfc$. Hence Lemma~\ref{l:basic-properties-classes}(1) yields $\mwsuzs\subseteq \Vfc$.

If $\m{A} = \m{A}_1 \oplus \m{A}_2$ with $\m{A}_1 \in [\Cn{n}]$ and $\m{A}_2 \in [\Cw]$ non-trivial, then, by Lemma~\ref{l:basic-ord-closure}(ii), $\Cn{n}\oplus \Cw \in \Vfc$. But, since $\Cw\oplus \Cw \in \Vfc$, by Lemma~\ref{l:basic-ord-closure}(iv), $\Cn{n} \oplus \Cw\oplus \Cw \in \Vfc$.
Hence Lemma~\ref{l:basic-properties-classes}(1) yields $\mwzs \subseteq \Vfc$.

If $\m{A} = \m{A}_1 \oplus \m{A}_2$ with $\m{A}_1 \in [\Cw]$ and $\m{A}_2 \in [\Cn{n}]$ non-trivial, then, by Lemma~\ref{l:basic-ord-closure}(ii), $\Cw\oplus \Cn{n} \in \Vfc$. But, since $\Cw\oplus \Cw \in \Vfc$, by Lemma~\ref{l:basic-ord-closure}(iv), $\Cw\oplus \Cw \oplus \Cn{n} \in \Vfc$.
Hence Lemma~\ref{l:basic-properties-classes}(1) yields $\mzsw \subseteq \Vfc$.

\item Suppose that  $\Vfc \in \bvarint{\Cn{n}, \Cw}$ and $\mwsuz \subsetneq \Vfc$. Then it follows similarly as in (d) that either $\mwsuzs \subseteq \Vfc$, $\mwsz \subseteq \Vfc$, or $\mzws \subseteq \Vfc$.

\item Suppose that  $\Vfc \in \bvarint{\Cn{n}, \Cw}$ and $\mwzs \subsetneq \Vfc$. Then, since $\mwzs$ contains all the members of $\mwzd$ with all components in $[\Cw]$ except possibly the first component in $[\Cn{n}]$, there is an $\m{A} \in \Vfc$ with $\m{A} = \m{A}_1 \oplus \m{A}_2$ such that $\m{A}_1 \in \mwuz$ and $\m{A}_2 \in [\Cn{n}]$ are non-trivial. Thus, by Lemma~\ref{l:basic-ord-closure}(ii), $\m{A}_1 \oplus \Cn{n} \in \Vfc$. There are two cases:

If $\m{A}_1 \in [\Cn{n}]$, then, by Lemma~\ref{l:basic-ord-closure}(ii), $\Cn{n} \oplus \Cn{n} \in \Vfc$. Moreover, since $\Cn{n} \oplus \Cw \oplus \Cw \in \Vfc$, Lemma~\ref{l:basic-ord-closure}(iv) yields $\Cn{n} \oplus \Cn{n} \oplus \Cw \oplus \Cw \in \Vfc$, i.e., by Lemma~\ref{l:basic-properties-classes}(3) $\mwszs\subseteq \Vfc$.

If $\m{A}_1 \in [\Cw]$, then, by Lemma~\ref{l:basic-ord-closure}(ii), $\Cw \oplus \Cn{n} \in \Vfc$. Moreover, since $\Cn{n} \oplus \Cw \in \Vfc$, Lemma~\ref{l:basic-properties-classes}(4) yields $\mwszs \subseteq \mwzd \subseteq \Vfc$.

\item Suppose that  $\Vfc \in \bvarint{\Cn{n}, \Cw}$ and $\mwsz \subsetneq \Vfc$. Then it follows similarly as in (f) that $\mwszs \subseteq \Vfc$.

\item Suppose that  $\Vfc \in \bvarint{\Cn{n}, \Cw}$ and $\mzsw \subsetneq \Vfc$ or $\mzws \subsetneq \Vfc$. Then it follows similarly as in (f) that $\mzsws \subseteq \Vfc$.

\item Suppose that  $\Vfc \in \bvarint{\Cn{n}, \Cw}$ and $\mwsuzs \subsetneq \Vfc$. Then, since the class $\mwsuzs$ contains all the members of $\mwzd$ with components only in $[\Cn{n}]$ or only in $[\Cw]$, there exists an $\m{A} \in \Vfc$ with $\m{A} = \m{A}_1 \oplus \m{A}_2$, where $\m{A}_1$ and $\m{A}_2$ are non-trivial, such that either $\m{A}_1 \in [\Cn{n}]$ and $\m{A}_2 \in [\Cw]$, or  $\m{A}_1 \in [\Cw]$ and $\m{A}_2 \in [\Cn{n}]$. 

In the first case, it follows from Lemma~\ref{l:basic-ord-closure}(ii) that $\Cn{n}\oplus \Cw \in \Vfc$. Moreover, since $\Cn{n}\oplus \Cn{n}, \Cw\oplus \Cw \in \Vfc$,  Lemma~\ref{l:basic-properties-classes}(3) yields $\mwszs\subseteq \Vfc$.

Similarly, it follows in the second case that $\mzsws \subseteq \Vfc$.

\item Suppose that  $\Vfc \in \bvarint{\Cn{n}, \Cw}$ and $\mwszs \subsetneq \Vfc$. Then there is an $\m{A} \in \Vfc$ such that $\m{A} = \m{A}_1 \oplus \m{A}_2$ with $\m{A}_1 \in [\Cw]$ and $\m{A}_2 \in [\Cn{n}]$ non-trivial. Thus, by Lemma~\ref{l:basic-ord-closure}(ii), $\Cw\oplus \Cn{n} \in \Vfc$. But also $\Cn{n}\oplus \Cw \in \Vfc$. So, by Lemma~\ref{l:basic-properties-classes}(4), $\mwzd \subseteq \Vfc$.

\item Suppose that  $\Vfc \in \bvarint{\Cn{n}, \Cw}$ and $\mzsws \subsetneq \Vfc$.  Then it follows similarly as in (j) that $\mwzd \subseteq \Vfc$.\qedhere
\end{enumerate}
\end{proof}

Because $\alg{\BHAP,\subseteq}$ and $\Omega(\BH)$ are isomorphic posets, the decomposition of $\alg{\BHAP,\subseteq}$ into intervals, as witnessed in the previous results, may also be understood as a decomposition of $\Omega(\BH)$ into intervals. Thus, we obtain the main result of this section.

\begin{thm}\label{t:BH-AP-char}
Every variety of basic hoops with the amalgamation property is either trivial or generated by one of the classes of totally ordered BL-algebras labeling one of the nodes in the intervals characterized by Lemmas~\ref{l:AP-basic-inteval1}--\ref{l:AP-basic-interval3}. In particular, $\Omega(\BH)$ is a union of countably many finite intervals.
\end{thm}

\begin{cor}
There are only countably infinitely many varieties of basic hoops with the amalgamation property.
\end{cor}

\section{Amalgamation in BL-algebras}\label{sec:BL-algebras}

We now consider $\Omega(\BL)$, using the previous section's work on basic hoops as a basis. Let $\V$ be a variety of BL-algebras. Recall that $\luk(\V)$ is the class of totally ordered MV-algebras in $\Vfsi$.
We denote by $\basic(\V)$ the class of totally ordered basic hoops $\m{A}_1$ such that $\m{A}_0 \oplus \m{A}_1 \in \Vfsi$ with $\m{A}_0$ a totally ordered MV-algebra. We define $\basicfc(\V) := \basic(\V)_\mathrm{fc}$ and note that $\Vfc \subseteq \luk(\V) \oplus \basicfc(\V)$.

Similarly to the previous section, our description of $\Omega(\BL)$ relies on the existing description of $\Omega(\MV)$, in tandem with our preceding classification of varieties of basic hoops with the amalgamation property. The next trio of results indicate how our reduction proceeds.

\begin{lemma}[{\cite[Proposition 1]{AB23}}]
If $\V\in\Omega(\BL)$, then $\vr(\luk(\V))\in\Omega(\MV)$.
\end{lemma}

\begin{lemma}
Let $\V$ be a variety of BL-algebras.
\begin{enumerate}[label = \textup{(\roman*)}]
\item $\hspu(\basic(\V)) = \basic(\V)$. 
\item If $\V\in\Omega(\BL)$, then $\vr(\basic(\V))\in\Omega(\BH)$.
\end{enumerate}
\end{lemma}
\begin{proof}
For (i) note that $\m{A} \in \basic(\V)$ if and only if $\m{2} \oplus \m{A} \in \V$. Thus, by Lemma~\ref{l:ISPU-ordinalsum}, $\hspu(\basic(\V)) = \basic(\V)$.

For (ii) Let  $\pair{i_1 \colon \m{A} \to \m{B},i_2\colon \m{A} \to \m{C}}$  be an essential span in $\basic(\V)$. If $\m{A}$ is trivial, then, since the span is essential, also  $\m{C}$ is trivial and clearly the span has an amalgam in $\basic(\V)$. Otherwise, note that, by Lemma~\ref{l:ess-ordinalsum}, the span  $\pair{i_1\colon \colon \m{A} \to \m{B}, i_2\colon \m{A} \to \m{C}}$ corresponds to an essential span of the form $\pair{\hat{i}_1\colon \m{2} \oplus\m{A} \to \m{2} \oplus\m{B},\hat{i}_2\colon \m{2} \oplus\m{A} \to \m{2} \oplus\m{C}}$ in $\Vfsi$. Since $\V$ has the AP, by Corollary~\ref{c:lin-ess-AP}, the span $\pair{\hat{i}_1,\hat{i}_2}$ has an amalgam $\pair{\hat{j}_1\colon \m{2} \oplus\m{B} \to \m{D}_0 \oplus\m{D}_1,\hat{j}_2\colon \m{2} \oplus\m{C} \to \m{D}_0 \oplus\m{D}_1}$ in $\Vfsi$, where
 $\m{D} = \m{D}_0 \oplus \m{D}_1$ with $\m{D}_0 \in \luk(\V)$ and $\m{D}_1 \in \basic(\V)$.  Since $\m{A}$ is non-trivial, $\pair{\hat{j}_1,\hat{j}_2}$ restricts to an amalgam $\pair{j_1 \colon \m{B} \to \m{D}_1, j_2 \colon \m{C} \to \m{D}_1}$ in $\basic(\V)$. Thus $\vr(\basic(\V))$ has the AP, by Corollary~\ref{c:lin-ess-AP}.
\end{proof}

\begin{cor}\label{c:AP-wajs-luk}
Let $\V \in \Omega(\BL)$. Then 
\begin{enumerate}[label = \textup{(\roman*)}]
\item $\luk(V) = [\m{A}]$ for some $\m{A} \in \{\Sm{m}, \Smw{m} \mid m \in \mathbb{N}\} \cup \{[0,1]_\MV \}$.
\item $\basicfc(\V)$ is in one of the intervals classified in Section~4. 
\end{enumerate}  
\end{cor}

As in the previous section, we first establish needed results regarding the classes of finite index chains of varieties in $\Omega(\BL)$.

\begin{lemma}\label{l:Vfc-regular-expr}
Let $\m{A}$ be a non-trivial totally ordered MV-algebra and $\V$ be a variety of basic hoops. Then $\vr([\m{A}] \oplus \Vfc)_\mathrm{fc} = [\m{A}] \oplus \Vfc$.
\end{lemma}

\begin{proof}
Note that $([\m{A}] \oplus \Vfsi)_\mathrm{fc} = [\m{A}] \oplus \Vfc$ and it is straightforward to check that $\hspu([\m{A}] \oplus \Vfsi) = [\m{A}] \oplus \Vfsi$. So we get $ [\m{A}] \oplus \Vfc \subseteq \vr([\m{A}] \oplus \Vfc)_\mathrm{fc}  \subseteq ([\m{A}] \oplus \Vfsi)_\mathrm{fc} = [\m{A}] \oplus \Vfc$.
\end{proof}

\begin{lemma}\label{l:AP-regularexpr-BL}
If $\m{A}$ is a non-trivial totally ordered MV-algebra and $\K$ is a class of totally ordered basic hoops with the (essential) amalgamation property, then $[\m{A}]\oplus \K$ has the (essential) amalgamation property.
\end{lemma}

\begin{proof}
Let $\pair{\phi_1 \colon \m{B} \to \m{C}, \phi_2 \colon \m{B} \to \m{D}}$ be an (essential) span in $[\m{A}]\oplus \K$. Then $\m{B} = \m{B}_0 \oplus \m{B}_1$,  $\m{C} = \m{C}_0 \oplus \m{C}_1$, and  $\m{D} = \m{D}_0 \oplus \m{D}_1$ with $\m{B}_0,\m{C}_0,\m{D}_0 \in [\m{A}]$ and $\m{B}_1,\m{C}_1,\m{D}_1 \in \K$. So the span $\pair{\phi_1,\phi_2}$ restricts to a span $\pair{\phi_1^0 \colon \m{B}_0 \to \m{C}_0, \phi_2^0 \colon \m{B}_0 \to \m{D}_0}$ in  $[\m{A}]$ and an (essential) span $\pair{\phi_1^1 \colon \m{B}_1 \to \m{C}_1, \phi_2^1 \colon \m{B}_1 \to \m{D}_1}$ in  $\K$. Since $[\m{A}]$ has the AP, by Lemma~\ref{l:one-chain}, and $\K$ has the (essential) AP by assumption, the spans have  amalgams  $\pair{\psi_1^0 \colon \m{C}_0 \to \m{E}_0, \psi_2^0 \colon \m{D}_0 \to \m{E}_0}$ in $[\m{A}]$ and $\pair{\psi_1^1 \colon \m{C}_1 \to \m{E}_1, \psi_2^1 \colon \m{D}_1 \to \m{E}_1}$ in $\K$, respectively. Thus the span $\pair{\phi_1,\phi_2}$ has an amalgam $\pair{\psi_1 \colon \m{C} \to \m{E}, \psi_2 \colon \m{D} \to \m{E}}$ in $[\m{A}] \oplus \K$ with $\m{E} = \m{E}_0 \oplus \m{E}_1$.
\end{proof}

\begin{lemma}\label{l:AP-union-Smw}
Let $\K$ be a class of totally ordered BL-algebras with the essential amalgamation property such that $\iso\sub(\K) = \K$ and  $[\Sm{m}]  = \luk(\K)\subseteq \K$ for some $m\geq 1$. Then $\K \cup [\Smw{m}]$ has the essential amalgamation property.
\end{lemma}

\begin{proof}
Note that $\K \cap [\Smw{m}] = [\Sm{m}]$ is essentially closed in $\K \cup [\Smw{m}]$, by Lemma~\ref{l:ess-closed-BL}(iii). Moreover, $\K$ and $[\Smw{m}]$ have the essential AP. Hence the claim follows from Lemma~\ref{l:ess-AP-union}.
\end{proof}

\begin{lemma}\label{l:AP-union-regexpr}
Let $\m{A}_1$ and $\m{A}_2$ be totally ordered MV-algebras such that $[\m{A}_1] \subseteq [\m{A}_2]$ is essentially closed in $[\m{A}_2]$, let $\m{B}_1$ and $\m{B}_2$ be independent Wajsberg chains, and $\K_i \in \set{[\m{B}_i], [\m{B}_i^\ast]}$ for $i=1,2$. Then $([\m{A}_1] \oplus \K_1) \cup ([\m{A}_2]\oplus \K_2)$ has the essential AP.
\end{lemma}

\begin{proof}
Note that, since $\m{B}_1$ and $\m{B}_2$ are independent, $([\m{A}_1] \oplus \K_1) \cap ([\m{A}_2]  \oplus \K_2)= [\m{A}_1]$ which is essentially closed in $[\m{A}_2]$ by assumption. Hence it is also essentially closed in $([\m{A}_1] \oplus \K_1) \cup ([\m{A}_2]\oplus \K_2)$, by Lemma~\ref{l:ess-ordinalsum}. Moreover, by Lemma~\ref{l:AP-2comp},  $\K_1$ and $\K_2$ have the essential AP. Hence Lemma~\ref{l:AP-regularexpr-BL} yields that $([\m{A}_1] \oplus \K_1)$ and $ ([\m{A}_2]\oplus \K_2)$ have the essential AP. But also, by Lemma~\ref{l:one-chain}, $[\m{A}_1]$ has the AP. Thus Lemma~\ref{l:ess-AP-union} yields that $([\m{A}_1] \oplus \K_1) \cup ([\m{A}_2]\oplus \K_2)$ has the essential AP.
\end{proof}

As for the basic hoop case, we will again need some closure properties to complete our classification. 

\begin{lemma}\label{l:ordsum-part-embed}
Let $\V\in\Omega(\BL)$ be such that $\m{A}_0 \oplus \m{A}_1 \in \Vfc$, where  $\m{A}_0$ is a non-trivial totally ordered MV-algebra and $\m{A}_1$ a totally ordered basic hoop. 
\begin{enumerate}[label = \textup{(\roman*)}]
\item If $\m{B}_0 \in \luk(\V)$ and $\m{A}_0 \leq \m{B}_0$ is an essential extension, then $\m{B}_0 \oplus \m{A}_1 \in \Vfc$.
\item If $\m{B}_0 \in \luk(\V)$ is simple, then $\m{B}_0 \oplus \m{A}_1 \in \Vfc$.
\item  If for $m\geq 1$, $\Smw{m} \in \luk(\V)$ and $\m{A}_0 \in [\Smw{m}] \setminus [\Sm{m}]$, then $\Smw{m} \oplus \m{A}_1 \in \Vfc$.
\item If $\m{B}_0 \oplus \m{B}_1 \in \Vfc$ such that $\m{A}_0 \leq \m{B}_0$, $\m{B}_1 \leq \m{A}_1$, and $\m{B}_1$ is non-trivial, then $\m{B}_0 \oplus \m{A}_1 \in \Vfc$.
\item If $\m{B}_0 \oplus \m{B}_1 \in \Vfc$ such that $\m{A}_1 \leq \m{B}_1$ is an essential extension, then $\m{A}_0 \oplus \m{B}_1 \in \Vfc$.
\end{enumerate}
\end{lemma}

\begin{proof}
First we note that, since $\V$ has the AP, $\Vfc$ has the essential AP, by Corollary~\ref{c:Vfc-AP}.

(i) Note that the span $\pair{\phi_1 \colon \m{A}_0 \to \m{A}_0 \oplus \m{A}_1, \phi_2 \colon \m{A}_0 \to \m{B}_0}$ with the obvious inclusion embeddings is essential. Thus it has an amalgam $\pair{\psi_1 \colon \m{A}_0 \oplus \m{A}_1 \to \m{C},\psi_2 \colon \m{B}_0 \to \m{C} }$ in $\Vfc$. Since $\psi_2(\m{B}_0)$ is  in the first component of $\m{C}$ it follows that $\m{C}$ contains a subalgebra that is isomorphic to $\m{B}_0 \oplus \m{A}_1$. Thus $\m{B}_0 \oplus \m{A}_1 \in \Vfc$.

(ii) Note that $\m{2} \leq \m{A}_0$, so $\m{2} \oplus \m{A}_1 \in \Vfc$. Moreover, $\m{2} \leq \m{B}_0$ is essential, since $\m{B}_0$ is simple. Hence, by part (i), $\m{B}_0 \oplus \m{A}_1 \in \Vfc$.

(iii) Since $\m{A}_0 \in [\Smw{m}] \setminus [\Sm{m}]$ it follows from \cite[Theorem 2.1]{Komori1981} that $\m{A}_0$ is infinite, so $\rad(\m{A}_0) \neq \{1\}$. Hence, $\rot{\Rad(\Smw{m})} \in \iso\sub\pu(\m{A}_0)$, by Corollary~\ref{c:ISPU-rotation}. But then, by Lemma~\ref{l:ISPU-ordinalsum}, $\rot{\Rad(\Smw{m})} \oplus \m{A}_1 \in \Vfc$. Now, since $\rot{\Rad(\Smw{m})} \leq \Smw{m}$ is essential, by Lemma~\ref{l:rad-ess}, it follows from part (i) that $\Smw{m} \oplus \m{A}_1 \in \Vfc$.

(iv) Note first that $\m{A}_0 \oplus \m{B}_1 \in \Vfc$, since $\m{B}_1 \leq \m{A}_1$. Moreover, since $\m{B}_1$ is non-trivial, the span $\pair{\phi_1 \colon \m{A}_0 \oplus \m{B}_1 \to \m{A}_0 \oplus \m{A}_1, \phi_2 \colon \m{A}_0 \oplus \m{B}_1 \to \m{B}_0 \oplus \m{B}_1}$ in $\Vfc$ with the obvious inclusion embeddings is essential. Thus it has an amalgam $\pair{\psi_1 \colon \m{A}_0 \oplus \m{A}_1 \to \m{C}, \psi_2 \colon  \m{B}_0 \oplus \m{B}_1 \to  \m{C}}$ in $\Vfc$. Since $\psi_2(\m{B}_0)$ is in the first component of $\m{C}$ it follows that $\m{C}$ contains a subalgebra that is isomorphic to $\m{B}_0 \oplus \m{A}_1$. Thus $\m{B}_0 \oplus \m{A}_1 \in \Vfc$.

(v) Note first that $\m{2} \oplus \m{A}_1 \in \Vfc$, since $\m{2} \leq \m{A}_0$. Moreover, by Lemma~\ref{l:ess-ordinalsum}, the span $\pair{\phi_1\colon \m{2} \oplus \m{A}_1 \to \m{A}_0 \oplus \m{A}_1,\phi_2 \colon \m{2} \oplus \m{A}_1 \to \m{B}_0 \oplus \m{B}_1}$ in $\Vfc$ is essential.  Thus it has an amalgam $\pair{\psi_1 \colon \m{A}_0 \oplus \m{A}_1 \to \m{C}, \psi_2 \colon  \m{B}_0 \oplus \m{B}_1 \to  \m{C}}$ in $\Vfc$. Since $\psi_1(\m{A}_0)$ is in the first component of $\m{C}$ it follows that $\m{C}$ contains a subalgebra that is isomorphic to $\m{A}_0 \oplus \m{B}_1$. Thus $\m{A}_0 \oplus \m{B}_1 \in \Vfc$.
\end{proof}

Recall that for a class of algebras $\K$ we denote by $\hm^+(\K)$ the class of non-trivial homomorphic images of members of $\K$.

\begin{lemma}\label{l:BL-pw-hom}
Let $\V\in\Omega(\BL)$. If $\bigoplus_{i=0}^n \m{A}_i \in \Vfc$, with $\m{A}_0$ a totally ordered MV-algebra and $\m{A}_i$ a Wajsberg chain for $i\geq 1$, then  $\hm^+(\m{A}_0) \oplus (\bigoplus_{i=1}^n \hm(\m{A}_i ))\subseteq \Vfc$.
\end{lemma}

\begin{proof}
We show that if $\m{B}$ is a non-trivial homomorphic image of $\m{A}_0$, then $\m{B} \oplus (\bigoplus_{i=1}^n \m{A}_i) \in \Vfc$ (for $\m{A}_i$ with $i\geq 1$ the proof is essentially the same as the proof of Lemma~\ref{l:basic-pw-hom}). By Lemma~\ref{l:radical} either $\m{B}$ is simple or $\rad(\m{B}) \neq \{1\}$. If $\m{B}$ is simple, then, by Lemma~\ref{l:ordsum-part-embed}(ii), $\m{B} \oplus (\bigoplus_{i=1}^n \m{A}_i) \in \Vfc$. Otherwise if $\rad(\m{B}) \neq \{ 1 \}$, then clearly also $\rad(\m{A}) \neq \{ 1 \}$. Thus, by Corollary~\ref{c:ISPU-rotation}, $\rot{\Rad(\m{B})}\in \iso\sub\pu(\m{A})$. But then, by Lemma~\ref{l:ISPU-ordinalsum}, $\rot{\Rad(\m{B})} \oplus (\bigoplus_{i=1}^n \m{A}_i) \in \Vfc$ and, since $\rot{\Rad(\m{B})} \leq \m{B}$ is essential, by Lemma~\ref{l:rad-ess}, it follows from Lemma~\ref{l:ordsum-part-embed}(i) that $\m{B} \oplus (\bigoplus_{i=1}^n \m{A}_i) \in \Vfc$.
\end{proof}

Thus Lemma~\ref{l:ISPU-ordinalsum} and Lemma~\ref{l:BL-pw-hom} yield the following result:

\begin{prop}\label{p:BL-pw-hspu}
Let $\V\in\Omega(\BL)$ and let $\m{A} \in \Vfc$ with $\m{A} = \bigoplus_{i=0}^n \m{A}_i$ of index $n+1$. Then $[\m{A}_0\m{A}_1 \cdots \m{A}_n] \subseteq \Vfc$.
\end{prop}

The following theorem provides an exhaustive description of varieties of BL-algebras with the amalgamation property, and is the main result of this study. It rests on the detailed description of $\Omega(\BH)$ given in Section~\ref{sec:hoops}.

\begin{thm}\label{t:AP-BL}
Let $\V$ be a variety of BL-algebras. Then $\V\in\Omega(\BL)$ if and only if 
$\basicfc(\V) \in \BHAP,$ and one of the following holds:
\begin{enumerate}[label = \textup{(\arabic*)}]
\item $\V$ is trivial.
\item $\luk(\V) = [\m{A}]$ for some $\m{A} \in \{\Sm{m}, \Smw{m} \mid m \in \mathbb{N}\} \cup \{[0,1]_\MV \}$  and $\Vfc = [\m{A}] \oplus \basicfc(\V)$.
\item $\luk(\V) = [\Smw{m}]$ for some $m\geq 1$ and $\Vfc = ([\Sm{m}] \oplus \basicfc(\V)) \cup [\Smw{m}]$.
\item $\luk(\V) = [\Smw{m}]$ for some $m\geq 1$, $\basicfc(\V) = \K_1 \cup \K_2$, where $\K_1 \in \{[\Cn{n}],[\Cn{n}^\ast]\}$ for some $n\geq 1$ and $\K_2 \in \{[\Cw], [\Cw^\ast]\}$; and either  $\Vfc = ([\Smw{m}] \oplus \K_1) \cup ([\Sm{m}] \oplus \K_2)$,
or  $\Vfc = ([\Sm{m}] \oplus \K_1) \cup ([\Smw{m}] \oplus \K_2)$.
\end{enumerate}
\end{thm}

\begin{proof}
The `only if' direction  follows from the Lemmas~\ref{l:AP-regularexpr-BL} -- \ref{l:AP-union-regexpr}.

For the other direction suppose that $\V$ is a variety of BL-algebras with the AP. Then, by Corollary~\ref{c:AP-wajs-luk}(i), there exists $\m{A} \in \{\Sm{m}, \Smw{m} \mid m \in \mathbb{N}\} \cup \{[0,1]_\MV \}$ such that $\luk(\V) = [\m{A}]$. If $\m{A} \in \{\Sm{m} \mid m \in \mathbb{N}\} \cup \{[0,1]_\MV \}$, then $\m{A}$ is simple and, by Lemma~\ref{l:ordsum-part-embed}(ii) together with Proposition~\ref{p:BL-pw-hspu}, $\Vfc = [\m{A}]\oplus \basicfc(\V)$. 
Otherwise, $\luk(\V) = [\Smw{m}]$ for some $m\geq 1$ and there are two cases:

Firstly, if the first component of each chain in $\Vfc$ with index greater than $1$ is in $[\Sm{m}]$ then it follows as above that $[\Sm{m}]\oplus \basicfc(\V) \subseteq \Vfc$ and  $([\Sm{m}]\oplus \basicfc(\V)) \cup [\Smw{m}] = \Vfc$.

Secondly, if  $\m{A}_0 \oplus \m{A}_1 \in \Vfc$ for $\m{A}_0 \in [\Smw{m}]\setminus [\Sm{m}]$ and $\m{A}_1$ non-trivial, then by Lemma~\ref{l:ordsum-part-embed}(iii), $\Smw{m}\oplus \m{A}_1$.  There are two subcases:
\begin{itemize}
\item If $\basicfc(\V) = [\m{B}\m{C}^\ast\m{D}^\ast\m{E}]$ (some of the component being possibly trivial), then, since $\m{A}_1$ is non-trivial, we may assume that $\m{A}_1$ is a non-trivial bounded or cancellative Wajsberg chain. Thus, it follows from Lemma~\ref{l:ordsum-part-embed}(iv) and Proposition~\ref{p:BL-pw-hspu} that $\Vfc = [\Smw{m}]\oplus \basicfc(\V) $.
\item Otherwise, by Corollary~\ref{c:AP-wajs-luk}(ii), it follows that $\basicfc(\V) = \K_1 \cup \K_2$ with $\K_1 \in \{[\Cn{n}], [\Cn{n}^\ast]\}$ and $\K_2 \in \{[\Cw], [\Cw^\ast]\}$. If $\m{A}_1 \in \K_1$, then, by Lemma~\ref{l:ordsum-part-embed}(iii) and Proposition~\ref{p:BL-pw-hspu}, $([\Smw{m}]\oplus \K_1) \cup ([\Sm{m}]\oplus \K_2) \subseteq \Vfc$. 
If there exists also a $\m{B}_0 \oplus \m{B}_1 \in \Vfc $ with $\m{B}_0 \in [\Smw{m}]\setminus [\Sm{m}]$ and $\m{B}_1 \in \K_2$ non-trivial, then, by Lemma~\ref{l:ordsum-part-embed}(iii), $[\Smw{m}]\oplus \m{A}_1 \in \Vfc$. So, by Lemma~\ref{l:ordsum-part-embed}(iii) and Proposition~\ref{p:BL-pw-hspu}, $([\Smw{m}]\oplus \K_1) \cup ([\Smw{m}]\oplus \K_2) = \Vfc$. If no such algebra is contained in $\Vfc$, then  $([\Smw{m}]\oplus \K_1) \cup ([\Sm{m}]\oplus \K_2) = \Vfc$.  

Similarly if $\m{A}_1 \in \K_2$ then either $([\Smw{m}]\oplus \K_1) \cup ([\Smw{m}]\oplus \K_2) = \Vfc$ or $([\Sm{m}]\oplus \K_1) \cup ([\Smw{m}]\oplus \K_2) = \Vfc$. \qedhere
\end{itemize}

\end{proof}

Note that from Theorem~\ref{t:AP-BL} together with the characterization of $\Omega(\BH)$ in Section~\ref{sec:hoops} it is straightforward to also obtain an explicit decomposition of $\Omega(\BL)$ into countably many finite intervals. However, some of these intervals are too large to nicely depict them in a figure.

\begin{cor}
There are only countably infinitely many varieties of BL-algebras with the amalgamation property.
\end{cor}

\section{Logical consequences}\label{sec:logic}

We now briefly summarize the logical ramifications of the results obtained in the foregoing work. Recall that if $\vdash$ is a consequence relation associated to some logic $\m{L}$, then $\vdash$ has the \emph{deductive interpolation property} provided that
\begin{quote}
whenever $\varphi\vdash_\m{L}\psi$, there exists a formula $\chi$, whose variables appear in both $\varphi$ and $\psi$, such that both $\varphi\vdash_\m{L}\chi$ and $\chi\vdash_\m{L}\psi$.
\end{quote}
It is well known that every axiomatic extension $\m{L}$ of $\m{BL}$ is algebraizable in the sense of Blok and Pigozzi, and that the equivalent algebraic semantics of any such $\m{L}$ is a subvariety $\V_{\m{L}}$ of $\BL$. Likewise, if $\m{L}$ is an axiomatic extension of the negation-free fragment of $\m{BL}$, then $\m{L}$ is algebraized by a subvariety $\V_{\m{L}}$ of basic hoops.

H\'{a}jek's basic logic $\m{BL}$ and its negation-free fragment are substructural logics with exchange, and for a logic $\m{L}$ subsumed under this umbrella it is well known (see, e.g., \cite{Metcalfe2014}) that $\m{L}$ has the deductive interpolation property if and only if the variety $\V_{\m{L}}$ has the amalgamation property. Consequently, the results of Section~\ref{sec:hoops} immediately entail the following result:
\begin{prop}\label{prop:BH DIP}
There are only countably many axiomatic extensions of the negation-free fragment of $\m{BL}$ that have the deductive interpolation property. These are exactly the logics corresponding to the varieties of basic hoops with the amalgamation property given in Proposition~\ref{p:intervals-basic}.
\end{prop}

For $\m{BL}$ itself the result follows from the work of Section~\ref{sec:BL-algebras}.

\begin{prop}\label{prop:BL DIP}
There are only countably many axiomatic extensions of $\m{BL}$ with the deductive interpolation property. These are exactly the logics corresponding to the varieties of BL-algebras with the amalgamation property given in Theorem~\ref{t:AP-BL}.
\end{prop}

In fact, for every extension of $\m{BL}$ and its negation-free fragment, the deductive interpolation property is equivalent to the strong deductive interprolation property as well as the (strong) Robinson property; see \cite{KihahaOno2010} for relevant definitions and implications among these metalogical properties. Consequently, we also obtain the strong deductive interpolation property and the strong Robinson property for each of the logics identified in Propositions~\ref{prop:BH DIP} and \ref{prop:BL DIP}.

\section{Concluding remarks}\label{sec:conclusion}

H\'{a}jek's basic logic sits at the intersection of two prominent families of substructural logics: On the one hand, it is a commutative logic within the broader environment of substructural logics validating the exchange rule and, on the other hand, it is a semilinear logic that is characterized by its totally ordered algebraic models. Both of these families of logics are sites of active research on interpolation, and we think the present study contributes to our understanding in each case.

First, the role of commutativity/exchange in interpolation has recently been the subject of quite a lot of work; see \cite{FMS2024,FG1,FG2}. Among other things, it is known that there are continuum-many substructural logics with the exchange rule that have the deductive interpolation property, but it is open whether this remains true with the addition of weakening \cite{FSLinear}. A priori, basic logic seems like a plausible place to look if one wants to investigate the latter question. However, the results here show that one must look elsewhere to find continuum-many logics with exchange and weakening that have deductive interpolation.

Second, deductive interpolation is now rather well understood among substructural logics characterized by linearly ordered algebraic models; see \cite{FSSurvey,FG1,FG2} and the references therein. The techniques developed here represent a significant contribution to the now-impressive array of tools for tackling interpolation for semilinear logics, and we anticipate that the general strategy we have used in the present paper applies in many other cases. In particular, we expect the essential amalgamation property may prove useful in studying amalgamation/interpolation in adjacent contexts---e.g., for varieties of MTL-algebras with tractable ordinal sum representations. We hope that the methodology developed here will one day serve as a component in understanding interpolation in substructural logics writ large.

\bibliographystyle{elsarticle-harv} 
\bibliography{literatur} 





\end{document}